\documentclass[a4paper,12pt,reqno]{amsart}
\usepackage[T1]{fontenc}
\usepackage[utf8]{inputenc}
\usepackage[english]{babel}
\usepackage{amsmath,amsthm,amssymb,mathabx}
\usepackage{fullpage}
\usepackage{bbold}	  
\usepackage{dsfont}   
\usepackage{float}    
\usepackage{xcolor}	  
\usepackage{graphicx} 
\usepackage{hyperref} 
\usepackage{fancyhdr}
\usepackage{tikz} 
\usetikzlibrary{matrix}
\usepackage{subfig}   
\advance\footskip0.5cm
\makeatletter
\def\@seccntformat#1{%
  \protect\textup{\protect\@secnumfont
    \ifnum\pdfstrcmp{subsection}{#1}=0 \bfseries\fi
    \csname the#1\endcsname
    \protect\@secnumpunct
  }%
}
\makeatother

\newcommand{\R}{\mathbb{R}}

\def\ip<#1>{\left\langle{#1}\right\rangle}

\newtheorem{theorem}{Theorem}[section]

\newtheorem{lemma}[theorem]{Lemma}
\newtheorem{corollary}[theorem]{Corollary}
\theoremstyle{definition}	
\newtheorem{definition}{Definition}[section]
\usepackage{todonotes}

\title[]{A sparse spectral method for Volterra integral equations using orthogonal polynomials on the triangle}
\author[T.\ S.\ Gutleb]{Timon~S.~Gutleb$^*$}\thanks{$^*$Department of Mathematics, Imperial College London, UK. (t.gutleb18@imperial.ac.uk)}
\author[S.\ \ Olver]{Sheehan~Olver$^\ddag$}\thanks{$^\ddag$Department of Mathematics, Imperial College London, UK. (s.olver@imperial.ac.uk)}
\date{\today}

\begin{document}

\begin{abstract}
    We introduce and analyse a sparse spectral method for the solution of Volterra integral equations using bivariate orthogonal polynomials on a triangle domain. The sparsity of the Volterra operator on a weighted Jacobi basis is used to achieve high efficiency and exponential convergence. The discussion is followed by a demonstration of the method on example Volterra integral equations of the first and second kind with known analytic solutions as well as an application-oriented numerical experiment. We prove convergence for both first and second kind problems, where the former builds on connections with Toeplitz operators. 
\end{abstract}

\maketitle
\thispagestyle{fancy}

\section{Introduction}

Define the \emph{Volterra integral operator}
\begin{equation}\label{eq:volterraintegral}
    (\mathcal{V}_K u)(x) := \int_0^{l(x)} K(x,y) u(y) \mathrm{d}y,
\end{equation}
where $K(x,y)$ is called the kernel, $u(y)$ is a given function of one variable and the limits of integration are either $l(x)=x$ or $l(x)=1-x$. This paper concerns Volterra integral equations of the first and second kind, that is, to find $u$ satisfying
\begin{equation*}
\mathcal{V}_K u = g \qquad\hbox{or}\qquad  ( I + \mathcal{V}_K) u = g.
\end{equation*}
Numerous applications and the fundamental nature of Volterra integral and integro-differential equations motivate research into efficient and accurate numerical solvers. Various forms of Volterra integral equations are analytically well-understood \cite{brunner_2017,pruss_evolutionary_2012,wazwaz_linear_2011}, have been the subject of various numerical approximation schemes \cite{brunner_2017,brunner_2004,babolian_direct_2008,maleknejad_numerical_2005}, and are encountered regularly in various scientific fields as well as engineering and finance applications \cite{brunner_2017,pruss_evolutionary_2012,van_den_bosch_pandemics_1999,wazwaz_linear_2011,krimer_non-markovian_2014,krimer_sustained_2016}.

 In this paper we present a method to compute Volterra integrals and solve Volterra integral equations by using orthogonal polynomials on a triangle domain \cite{dunkl_orthogonal_2014,olver_sparse_2019} to both resolve the kernel and to reduce the equations to banded linear systems. The method is in the same spirit as some previous contributions to the field of numerical Volterra, Fredholm, singular integral and differential equations based on operators and orthogonal polynomials such as \cite{akyuz-dascioglu_chebyshev_2006,koroglu_chebyshev_1998,slevinsky_singular_2017,hale_ultraspherical_2018} but differs  in choice of basis and domain, leading to operator bandedness properties which can be exploited for significantly increased efficiency. Notably the approach introduced in this paper can be used for a wider range of kernels than many other Volterra integral equation solvers such as the methods based on orthogonal polynomials due to Loureiro and Xu \cite{loureiro_volterra-type_2019,xu_spectral_2018}, the recently developed ultraspherical spectral method in \cite{hale_ultraspherical_2018} or the Fourier extension method in \cite{xu_fast_2017} as it is not limited to convolution kernel cases, that is kernels of the form $K(x,y) = K(x-y)$, but works for a wider class of kernels.

The sections in this paper are organized as follows: Section \ref{sec:functionapproximation} introduces the required aspects of univariate and bivariate polynomial function approximation on a real interval and the triangle respectively. Section \ref{sec:volterra} introduces an efficient numerical method for  Volterra integrals and integral equations and discusses how to approach kernel computations using a multivariate variant of Clenshaw's algorithm. In Section \ref{sec:numericalexamples} we show the scheme in action in both toy and application-based examples. Proofs of convergence for well-posed problems are discussed in Section \ref{sec:analysis}.
\section{Function approximation with orthogonal polynomials} \label{sec:functionapproximation}
\subsection{Jacobi polynomials on the real interval} \label{sec:jacobirealinterval}
Multivariate orthogonal polynomials are ordered sets of polynomials satisfying a particular pair-wise and weighted orthogonality condition, often of the form 
\begin{equation} \label{eq:innerprod}
\ip<P_{m,k},P_{n,j}> = \int_\Omega P_{m,k}(\mathbf{x}) P_{n,j}(\mathbf{x}) W(\mathbf{x}) \mathrm{d}A = C \delta_{mn} \delta_{jk},
\end{equation}
where $C \neq 0$ and $P_{m,k}$ are total degree $m$ polynomials. 
Many such sets of orthogonal polynomials are well-known and well-studied on various domains $\Omega$ such as $\R$, real intervals, simple $2$D and $3$D domains, as well as various higher dimensional spheres and polygons \cite{dunkl_orthogonal_2014}. The relevant set of orthogonal polynomials for this paper are the Jacobi polynomials on the real line and on the triangle respectively. This section will thus give a quick overview of Jacobi polynomials aimed at equipping us with the tools needed to develop the Volterra integral equation solvers in later sections. We refer to \cite{dunkl_orthogonal_2014,gautschi_orthogonal_2004} for introductions with broader scope.\\

The Jacobi polynomials  are orthogonal on $[-1,1]$:
\begin{equation*}
  \int_{-1}^{1} C_{(\alpha,\beta,m,n)} \left( 1-x \right)^\alpha\left( 1+x \right)^\beta P_m^{(\alpha,\beta)}(x) P_n^{(\alpha,\beta)}(x) \mathrm{d}x = \delta_{nm},
\end{equation*}
where $W_{(\alpha, \beta)}(x)=C_{(\alpha,\beta,m,n)}\left ( 1-x \right)^\alpha\left( 1+x \right)^\beta$ acts as the weight function and $\delta_{nm}$ is the Kronecker delta. While the choice of $[-1,1]$ is natural, the Jacobi polynomials can be shifted to any real interval an application requires. For $\alpha=\beta=0$ the Jacobi polynomials reduce to the Legendre polynomials \cite{dunkl_orthogonal_2014}. \\

One of the primary applications of interest for the study of orthogonal polynomials are their applications in the expansion of non-polynomial functions:
\begin{equation*}
    f(x) = \sum_{n=0}^\infty p_n(x) f_n = \mathbf{P}(x)^\mathsf{T} \mathbf{f},
\end{equation*}
where $f_n$ is the function-specific coefficient of the $n$-th polynomial $p_n$ and we use the notation 
\begin{align*}
     \mathbf{P}(x) := \begin{pmatrix}
           p_{0}(x) \\
           p_{1}(x) \\
           \vdots
         \end{pmatrix} &,\hspace{5mm}
     \mathbf{f} := \begin{pmatrix}
           f_{0} \\
           f_{1} \\
           \vdots
         \end{pmatrix}.
\end{align*}
For numerical applications one uses finitely many terms in the above sum to obtain an approximation. If a distinction between different sets of polynomials and coefficient vectors on different domains is required we specify by indicating the type of polynomials using standard notation for the polynomials, such as $\mathbf{P}^{(\alpha,\beta)}(x)$ for the Jacobi polynomials on a real interval, and the domain using index notation, e.g. for the bivariate orthogonal polynomial coefficient vector of $g(x,y)$ on the triangle domain we write $\mathbf{g}_\vartriangle$.\\

To use function approximation of this type in a non-trivial numerical application one needs ways to do computations on functions represented as coefficient vectors. Basic computations such as addition and subtraction of functions have obvious implementations. Furthermore one can compute $x f(x)$ if $f(x)$ is already approximated as a coefficient vector: to do this one uses so-called Jacobi operators $\mathrm{J}$ which act as 
\begin{equation*}
    \mathbf{P}(x)^\mathsf{T} \mathrm{J} \mathbf{f}_{[0,1]} = x f(x).
\end{equation*}
This is efficiently possible because the Jacobi polynomials satisfy a three-term recurrence relationship, making $\mathrm{J}$ a tridiagonal operator (see e.g. \cite{dunkl_orthogonal_2014,nist_2018,olver_sparse_2019}):
\begin{equation}\label{eq:jacobiop_realline}
\mathrm{J} = \begin{pmatrix} a_0 & b_0 & &  \\ c_0 & a_1 & b_1 & \\ & c_1 & a_2 & \ddots \\ &&\ddots&\ddots \end{pmatrix}.
\end{equation}
Additionally, our approach to Volterra integral equations of the second kind will require explicit constructors for raising operators $\mathrm{S}_{(\alpha, \beta)}^{(\alpha + 1,\beta)}, \mathrm{S}_{(\alpha, \beta)}^{(\alpha,\beta +1)}$ which are defined to increment from the Jacobi bases $\mathbf{P}^{(\alpha,\beta)}(x)$ to $\mathbf{P}^{(\alpha +1,\beta)}(x)$ and $\mathbf{P}^{(\alpha,\beta +1)}(x)$ respectively. Increments to $\alpha$ and $\beta$ can be computed using these operators but decrementing is generally only well-defined in the sense of \emph{weighted} lowering operators:
\begin{align*}
x f(x) &= \mathbf{P}^{(\alpha-1,\beta)}(x){}^\mathsf{T} \mathrm{L}_{(\alpha, \beta)}^{(\alpha-1,\beta)} \mathbf{f},\\
(1-x) f(x) &= \mathbf{P}^{(\alpha,\beta-1)}(x){}^\mathsf{T} \mathrm{L}_{(\alpha, \beta)}^{(\alpha,\beta-1)} \mathbf{f}.
\end{align*}
 The explicit forms of the operators $\mathrm{J}$, $\mathrm{S}_{(\alpha, \beta)}^{(\alpha + 1,\beta)}$, $\mathrm{S}_{(\alpha, \beta)}^{(\alpha,\beta +1)}$, $\mathrm{L}_{(\alpha, \beta)}^{(\alpha-1,\beta)}$ and $\mathrm{L}_{(\alpha, \beta)}^{(\alpha,\beta-1)}$ are well known in the literature, see for example \cite{nist_2018,olver_sparse_2019,dunkl_orthogonal_2014} and the references therein.
\subsection{Jacobi polynomials on the triangle} \label{sec:jacobitriangle}
We now briefly discuss how function approximation using bivariate orthogonal polynomials works in general and then move on to discuss the Jacobi polynomials on the canonical unit simplex 
\begin{equation*}
T^2 = \left\{ (x,y) : 0 \leq x, 0 \leq y \leq 1-x  \right\}.
\end{equation*}
We use a basis on this triangle in the following sections to compute Volterra integrals and solve integral equations. As in the univariate case, bivariate orthogonal polynomials are said to be orthogonal with respect to an inner product akin to (\ref{eq:innerprod}).

Analogously to how functions of a single variable may be expanded into a basis of univariate orthogonal polynomials as $f(x) = \sum_{n=0}^\infty p_n(x) f_n$ we can expand a function of two variables in a basis of bivariate polynomials as
\begin{equation*}
    f(x,y) = \sum_{n=0}^\infty \sum_{k=0}^n p_{n,k}(x,y) f_{n,k}.
\end{equation*}
Writing the bivariate polynomials of total degree $n$ as
\begin{equation*}
    \mathbb{P}_n(x,y)= \begin{pmatrix}
           p_{n,0}(x,y) \\
           p_{n,1}(x,y) \\
           \vdots \\
           p_{n,n}(x,y)
           \end{pmatrix}
\end{equation*}
allows for the following compact notation for the infinite-dimensional polynomial basis:
\begin{equation*}
\mathbf{P}(x,y)=\begin{pmatrix}
           \mathbb{P}_{0}(x,y) \\
           \mathbb{P}_{1}(x,y) \\
           \vdots \\
           \end{pmatrix}.
\end{equation*}
In this notation the expansion of a function of two variables in the bivariate polynomial basis becomes
\begin{equation*}
    f(x,y) = \sum_{n=0}^\infty \sum_{k=0}^n p_{n,k}(x,y) f_{n,k} = \mathbf{P}(x,y)^\mathsf{T} \mathbf{f}.
\end{equation*}
For function approximation one simply uses an appropriate finite cutoff of this expansion. 

On the triangle $T^2$ we focus on the Jacobi weights $x^\alpha y^\beta (1-x-y)^\gamma$. One elegant way to define the corresponding Jacobi polynomials $\mathbf{P}^{(\alpha, \beta, \gamma)} (x,y)$ on the canonical triangle $T^2$ is by referring to the Jacobi polynomials $\mathbf{P}^{(\alpha,\beta)} (x)$ on the real interval $[-1,1]$ (compare \cite[Proposition 2.4.1]{dunkl_orthogonal_2014}):
\begin{equation}\label{eq:jacobitriangleeqs}
    P_{k,n}^{(\alpha, \beta, \gamma)} (x,y) = \left(1-x\right)^k P_{n-k}^{\left(2k+\beta+\gamma+1,\alpha\right)}\left(2x-1\right) P_k^{\left(\gamma,\beta\right)}\left(\frac{2y}{1-x}-1\right).
\end{equation}
Defined as such the triangle Jacobi polynomials are orthogonal with respect to a weighted integral over the canonical triangle domain $T^2$:
\begin{equation*}
    \int_{0}^1 \int_{0}^{1-x} x^\alpha y^\beta (1-x-y)^\gamma P_{k,n}^{(\alpha, \beta, \gamma)} (x,y) P_{j,m}^{(\alpha, \beta, \gamma)} (x,y) \mathrm{d}y \mathrm{d}x = C_{(\alpha,\beta,\gamma)} \delta_{jk} \delta_{mn}.
\end{equation*}
The detailed form of the constant $C_{(\alpha,\beta,\gamma)}$ is not important here but can for example be found in \cite{dunkl_orthogonal_2014}. We will primarily use the Jacobi polynomials shifted to the $[0,1]$ interval and denote them by $\tilde{\mathbf{P}}^{(\alpha,\beta)}$, which allows us to write the Jacobi polynomials on the triangle as:
\begin{equation}\label{eq:shiftedjacobitriangleeqs}
    P_{k,n}^{(\alpha, \beta, \gamma)} (x,y) = (1-x)^k \tilde{P}_{n-k}^{\left(2k+\beta+\gamma+1,\alpha\right)}\left(x\right) \tilde{P}_k^{\left(\gamma,\beta\right)}\left(\frac{y}{1-x}\right).
\end{equation}
As in the 1-dimensional case we can define Jacobi operators $\mathrm{J}_x$ and $\mathrm{J}_y$, one for each variable, which respectively act as
\begin{align*}
     \mathbf{P}(x,y)^\mathsf{T} \mathrm{J}_x \mathbf{f}_\vartriangle &= x f(x,y),\\
     \mathbf{P}(x,y)^\mathsf{T} \mathrm{J}_y \mathbf{f}_\vartriangle &= y f(x,y),
\end{align*}
for a given bivariate polynomial basis. Unlike the 1-dimensional Jacobi polynomial case these operators are not tridiagonal but block tridiagonal operators for the triangle Jacobi polynomials \cite{olver_sparse_2019}:
\begin{equation} \label{eq:jacobiop_triangle}
\mathrm{J}_x = \begin{pmatrix} A_0^x & B_0^x & &  \\ C_0^x & A_1^x & B_1^x & \\ & C_1^x & A_2^x & \ddots \\ &&\ddots&\ddots \end{pmatrix}, \quad \mathrm{J}_y = \begin{pmatrix} A_0^y & B_0^y & &  \\ C_0^y & A_1^y & B_1^y & \\ & C_1^y & A_2^y & \ddots \\ &&\ddots&\ddots \end{pmatrix},
\end{equation}
where $A_n^x,A_n^y \in \mathbb{R}^{(n+1)\times(n+1)}$, $B_n^x,B_n^y \in \mathbb{R}^{(n+1)\times(n+2)}$ and $C_n^x,C_n^y \in \mathbb{R}^{(n+2)\times(n+1)}$. Analogous operators to the raising and lowering operators discussed for the real interval case can be constructed for the Jacobi polynomials on the triangle as well, see \cite{olver_recurrence_2019,olver_sparse_2019}, but we omit their discussion as we will not make direct use of them in this paper.\\
To make use of Jacobi polynomials for the approximation of functions on the triangle domain in a numerical context one  requires efficient algorithms to determine the coefficient vector $\mathbf{f}_\vartriangle$ for a given function $f(x,y)$ of two variables. This can be done using an algorithm and its implementation in a C library by Slevinsky \cite{slevinsky_conquering_2017,slevinsky_fast_2017,slevinsky_fasttransforms_2019}.
\subsection{Function evaluation using Clenshaw's algorithm}\label{sec:clenshawintro}
Clenshaw's algorithm provides an efficient and direct method to evaluate functions expanded into orthogonal polynomial bases at given points, i.e. to evaluate $\sum_{n=0}^{N} p_n(\mathbf{x}) f_n$ at $\mathbf{x}_* \in \mathbb{R}^d$, cf. \cite{clenshaw_note_1955,olver_sparse_2019}. The algorithm makes use of the polynomial basis' recurrence relationships to reduce function evaluation to the solution of an upper triangular linear system using backward substitution. In this section we give an outline of how this is done for Jacobi polynomials on the real interval and the triangle, which is discussed in more detail in \cite{olver_sparse_2019}. An operator valued variant of what is discussed in this section will be used for efficient kernel computations for Volterra integrals in section \ref{sec:kernelcomp}.  We mention a major benefit of Clenshaw's algorithm over building polynomials/operators via forward recurrences is that there is substantially less memory needed in the intermediary calculations. \\

For the case of Jacobi polynomials on a real interval, the three-term recurrence relationship seen in the Jacobi operator in (\ref{eq:jacobiop_realline}) can be used to write
\begin{equation}
\mathcal{L}_N(x_*) \mathbf{P}_N^{(\alpha,\beta)}(x_*) = \begin{pmatrix} 1 &  & & &&  \\ a_0-x_* & b_0 & & &  \\ c_0 & a_1-x_* & b_1 & &  \\&\ddots&\ddots&  \ddots& \\   & & c_{N-2} & a_{N-1}-x_* & b_{N-1}   \end{pmatrix} \begin{pmatrix} P^{(\alpha,\beta)}_1(x_*) \\ P^{(\alpha,\beta)}_2(x_*) \\ P^{(\alpha,\beta)}_3(x_*)\\ \vdots \\ P^{(\alpha,\beta)}_N(x_*) \end{pmatrix} = \mathbf{e}_0,
\end{equation}
where $\mathbf{e}_0$ is the first standard basis vector with $1$ in its first component and of appropriate length. Solving this lower triangular system via forward substituition provides a way to recursively evaluate each component of $\mathbf{P}^{(\alpha,\beta)}(x)$ and thus also $\mathbf{P}^{(\alpha,\beta)}(x)^\mathsf{T} \mathbf{f}$ if the coefficients of $f(x)$ in this basis are known. Clenshaw's algorithm is conceptually similar but uses backward substition on the system 
\begin{equation}
f(x_*) = \mathbf{P}^{(\alpha,\beta)}_N(x_*)^\mathsf{T} \text{\textbf{\textit{a}}} = \mathbf{e}_0^\mathsf{T} \mathcal{L}_N(x_*)^{-\mathsf{T}} \text{\textbf{\textit{a}}},
\end{equation}
where \textbf{\textit{a}} is the column vector collecting $a_0$ to $a_N$. The case for the Jacobi polynomials on the triangle was recently discussed in \cite{olver_sparse_2019} and on the basis of the recurrence in (\ref{eq:jacobiop_triangle}) involves a block triangular system for evaluation at $\mathbf{x}_* = (x_*,y_*)$ instead:
\begin{equation*}
\mathcal{L}_N(\mathbf{x}_*) \mathbf{P}_N^{(\alpha,\beta,\gamma)}(\mathbf{x}_*) = \begin{pmatrix} \mathbb{1}_1 &  & & &  \\ A_0^x-x_*\mathbb{1}_1 & B_0^x & &  \\ A_0^y-y_*\mathbb{1}_1 & B_0^y & &   \\ C_0^x & A_1^x-x_*\mathbb{1}_2 & B_1^x &  \\ C_0^y & A_1^y-y_*\mathbb{1}_2 & B_1^y &  \\&\ddots&\ddots&  \ddots \end{pmatrix} \mathbf{P}_N^{(\alpha,\beta,\gamma)}(\mathbf{x}_*) = \mathbf{e}_0,
\end{equation*}
where $\mathbb{1}_k$ denotes the $k\times k$ identity matrix. As this is not a triangular but a block triangular matrix one cannot use forward substitution without first applying a preconditioner:
\begin{equation*}
\begin{pmatrix} 1 & & &  \\  & B_0^+ &  & \\ & &  B_1^+& \\ && &\ddots \end{pmatrix}\mathcal{L}_N(\mathbf{x}_*) = \mathrm{\tilde{\mathcal{L}}}_N(\mathbf{x}_*).
\end{equation*}
$\tilde{\mathcal{L}}_N(\mathbf{x}_*)$ is then a proper lower triangular matrix and can be used in an analogous system to the ones above to evaluate the polynomials, and thus a function expanded into that polynomial basis, recursively via forward substitution. A preconditioner which satisfies these requirements is the block diagonal matrix whose elements are comprised of a left inverse of the blocks 
\begin{equation*}
B_n = \begin{pmatrix} B_n^x \\ B_n^y \end{pmatrix},
\end{equation*}
such that $B_n^+ B_n = \mathbb{1}_n$. Clenshaw's algorithm for the triangle Jacobi polynomials is thus
\begin{equation*}
f(\mathbf{x}_*) = \mathbf{P}^{(\alpha,\beta,\gamma)}_N(\mathbf{x}_*)^\mathsf{T} \text{\textbf{\textit{A}}} = \mathbf{e}_0^\mathsf{T} \tilde{\mathcal{L}}_N(\mathbf{x}_*)^{-\mathsf{T}} \text{\textbf{\textit{A}}}.
\end{equation*}
This system can be solved via backward substitution in optimal $O(N^2)$ complexity if one chooses $B_n^+$ carefully, see \cite{olver_sparse_2019}.
\section{A numerical method for Volterra integral equations} \label{sec:volterra}
\subsection{Volterra integrals on the triangle}
In this section we describe how to represent Volterra integrals using bivariate orthogonal polynomials on a triangle domain by moving to a view of operators acting on coefficient vectors. The following section extends this method to Volterra integral equations of the first and second kind.\\

We first  describe the idea behind the relevant operators and their use before determining their entries in matrix representation. The first operator we need is the integration operator for a function given as the coefficients of orthogonal polynomials on a triangle. We label this operator $\mathrm{Q}_y$ and it acts as
\begin{equation*}
    \mathbf{P}(x)^\mathsf{T} \mathrm{W}_\mathrm{Q} \mathrm{Q}_y \mathbf{f}_\vartriangle = \int_{0}^{1-x} f(x,y) \mathrm{d}y,
\end{equation*}
where $\mathrm{W}_\mathrm{Q}$ is a to-be-determined weight function which depends on the used basis. The reason for the limits of integration to be defined in this way for $\mathrm{Q}_y$ will become clear once we discuss the explicit form of these operators and how one can make optimal use of the triangle domain's symmetries.
Second, we need an operator $\mathrm{E}_y$ which extends a one-dimensional function on $[0,1]$ to one on $T^2$, that is:
\begin{equation*}
    \mathbf{P}(x)^\mathsf{T}\mathbf{f}_{[0,1]}  =     \mathbf{P}(x,y)^\mathsf{T}  \mathrm{E}_y \mathbf{f}_{[0,1]} 
\end{equation*}
Together these two operators can be used to compute integrals of the form
\begin{equation*}
    \int_0^{1-x} f(y) \mathrm{d}y = \mathbf{P}(x)^\mathsf{T} \mathrm{W}_\mathrm{Q}  \mathrm{Q}_y \mathrm{E}_y \mathbf{f}_{[0,1]}
\end{equation*}
with function $f$ depending on a single variable. To instead integrate from $0$ to $x$ we use a reflection operator. Due to symmetries of the polynomials, particular basis changes in a Jacobi basis obey the simple rule \cite{nist_2018, dunkl_orthogonal_2014}:
\begin{equation*}
\tilde{P}_n^{(\alpha,\beta)}(x) = (-1)^n \tilde{P}_n^{(\beta,\alpha)}(1-x).
\end{equation*}
We use $\mathrm{R}$ to refer to the operator that uses the above property to reflect the function on the $[0,1]$ interval via a basis change, i.e.
\begin{equation}\label{eq:reflectbasis}
\tilde{\mathbf{P}}^{(\alpha,\beta)}(x)^\mathsf{T} \mathrm{R} \mathbf{f}  = \sum_n (-1)^n f_n \tilde{P}_n^{(\beta,\alpha)}(x)  =  f(1-x).
\end{equation}
$\mathrm{J}_x$ and $\mathrm{J}_y$ have important commutation relations with the introduced $\mathrm{Q}_y$ and $\mathrm{E}_y$ operators. As the $\mathrm{Q}_y$ operator integrates with respect to $y$ and collapses a bivariate coefficient vector back to a univariate one the multiplication-with-$x$ operator changes from being multiplication-with-$x$ on the triangle ($=\mathrm{J}_x$) to being multiplication-with-$x$ on the real interval ($=\mathrm{J}$) when pulled through the $\mathrm{Q}_y$ operator. A similar relation holds for similar reasons for $\mathrm{J}_y$ and $\mathrm{E}_y$:
\begin{align}\label{eq:commutation1}
     \mathrm{Q}_y \mathrm{J}_x \mathbf{f}_\vartriangle &= \mathrm{J} \mathrm{Q}_y  \mathbf{f}_\vartriangle,\\ \label{eq:commutation2}
      \mathrm{J}_y \mathrm{E}_y \mathbf{f}_{[0,1]} &=  \mathrm{E}_y \mathrm{J} \mathbf{f}_{[0,1]}. 
\end{align}
We now give the explicit matrix representations for the operators $\mathrm{Q}_y$ and $\mathrm{E}_y$ and discuss a sensible polynomial basis choice. The explicit form of the Jacobi operators on the real line is known in the literature (e.g. \cite{dunkl_orthogonal_2014,olver_sparse_2019}) and thus receives no further discussion here. To determine the explicit form of $\mathrm{Q}_y$ we begin by plugging in the polynomial expansion of $f(x,y)$ into the intended integral operation and using the Jacobi polynomials on the triangle domain as seen in  (\ref{eq:shiftedjacobitriangleeqs}) for our basis $p_{n,k}$ with $\alpha=\beta=\gamma=0$:
\begin{align*}
    \mathbf{P}^{(1,0)}(x)^\mathsf{T} \mathrm{W}_\mathrm{Q} \mathrm{Q}_y \mathbf{f}_\vartriangle &= \int_{0}^{1-x} f(x,y) \mathrm{d}y = \int_{0}^{1-x} \sum_{n=0}^\infty \sum_{k=0}^n p_{n,k}(x,y) f_{n,k} \mathrm{d}y  \\ &= \sum_{n=0}^\infty \sum_{k=0}^n f_{n,k} (1-x)^k \tilde{P}_{n-k}^{(2k+1,0)}(x) \int_{0}^{1-x} \tilde{P}_{k}^{(0,0)}\left(\frac{y}{1-x}\right) \mathrm{d}y \\ &= \sum_{n=0}^\infty \sum_{k=0}^n f_{n,k} (1-x)^{k+1} \tilde{P}_{n-k}^{(2k+1,0)}(x) \int_{0}^{1} \tilde{P}_{k}^{(0,0)}\left(s\right) \mathrm{d}s,
\end{align*}
where a substitution of $\frac{y}{1-x} \rightarrow s$ was made in the last step. As $\tilde{P}_{k}^{(0,0)}$ are just the Legendre polynomials on $[0,1]$ we see that $\int_{0}^{1} \tilde{P}_{k}^{(0,0)}\left(s\right) \mathrm{d}s = 0, \forall k>0$ and $\int_{0}^{1} \tilde{P}_{0}^{(0,0)}\left(s\right) \mathrm{d}s = 1$, resulting in
\begin{align*}
    \mathbf{P}^{(1,0)}(x)^\mathsf{T} \mathrm{W}_\mathrm{Q} \mathrm{Q}_y \mathbf{f}_\vartriangle &=  \sum_{n=0}^\infty f_{n,0} (1-x) \tilde{P}_{n}^{(1,0)}(x)
\end{align*}
for integration from $0$ to $1-x$. Via (\ref{eq:reflectbasis}) we further obtain
\begin{align*}
    \mathbf{P}^{(0,1)}(x)^\mathsf{T} \mathrm{W}_\mathrm{Q} \mathrm{Q}_y \mathbf{f}_\vartriangle &=  \sum_{n=0}^\infty f_{n,0} (-1)^n (1-x) \tilde{P}_{n}^{(1,0)}(x)
\end{align*}
for integration from $0$ to $x$. This derivation shows that starting in the Jacobi polynomial basis on the triangle $T^2$ with $\alpha=\beta=\gamma=0$ for the approximation of $f(x,y)$ results in the following block diagonal structure for the integration from $0$ to $1-x$ operator with weight $\mathrm{W}_\mathrm{Q}=(1-x)$:
\begin{equation*}
\mathrm{Q}_y =   \left(\begin{array}{cccccccccc}
\cline{1-1}
\multicolumn{1}{|c|}{1} &  &  &  &  &  &  &  &  &  \\ \cline{1-3}
\multicolumn{1}{c|}{} & 1 & \multicolumn{1}{c|}{0} &  &  &  &  &  &  &  \\ \cline{2-6}
 &  & \multicolumn{1}{c|}{} & 1 & 0 & \multicolumn{1}{c|}{0} &  &  &  &  \\ \cline{4-10} 
 &  &  &  &  & \multicolumn{1}{c|}{} & \ddots & \ddots & \ddots & \multicolumn{1}{c|}{\ddots} \\ \cline{7-10} 
\end{array}\right)
\end{equation*}
where the $n$-th block is an $n$-dimensional row vector with $1$ in the first element and $0$ in all remaining elements. An additional $(-1)^n$ term and change of basis changes this integration to be from $0$ to $x$ instead. The expansion operator $\mathrm{E}_y$ from the $\mathbf{P}^{(1,0)}(x)$ basis to the canonical triangle Jacobi polynomials where $\alpha=\beta=\gamma=0$ has the block diagonal structure
\begin{equation*}
\mathrm{E}_y = \left( \begin{array}{cccc}
\cline{1-1}
\multicolumn{1}{|c|}{\times} &  &  &  \\ \cline{1-2}
\multicolumn{1}{c|}{} & \multicolumn{1}{c|}{\times} &  &  \\
\multicolumn{1}{c|}{} & \multicolumn{1}{c|}{\times} &  &  \\ \cline{2-3}
 & \multicolumn{1}{c|}{} & \multicolumn{1}{c|}{\times} &  \\
 & \multicolumn{1}{c|}{} & \multicolumn{1}{c|}{\times} &  \\
 & \multicolumn{1}{c|}{} & \multicolumn{1}{c|}{\times} &  \\ \cline{3-4} 
 &  & \multicolumn{1}{c|}{} & \multicolumn{1}{c|}{\ddots} \\
 &  & \multicolumn{1}{c|}{} & \multicolumn{1}{c|}{\ddots} \\
 &  & \multicolumn{1}{c|}{} & \multicolumn{1}{c|}{\ddots} \\
 &  & \multicolumn{1}{c|}{} & \multicolumn{1}{c|}{\ddots} \\ \cline{4-4} 
\end{array} \right)
\end{equation*}
where the $n$-th block is an $n$-dimensional column vector whose $j$-th entry is given by
\begin{equation*}
\frac{(-1)^{j+n}(2j-1)}{n}.
\end{equation*}
Importantly, multiplication of $\mathrm{Q}_y$ and $\mathrm{E}_y$ yields a diagonal matrix whose $n$-th entry can be directly generated without any matrix multiplication being required (compare \cite{nist_2018}):
\begin{equation*}\label{eq:diagonal}
(\mathrm{Q}_y\mathrm{E}_y)_{n,n} =(\mathrm{D}_y)_{n,n}=\frac{(-1)^{n+1}}{n}.
\end{equation*}
These observations justify the basis choices as well as the choice of the limits of integration for $\mathrm{Q}_y$ from the standpoint of computational efficiency. Defining $\mathrm{Q}_y$ as the integration operator from $0$ to $x$ does not avoid the reflection step and only results in a less efficient or equivalent placement for it.\\
\subsection{Kernel computations using Clenshaw's algorithm} \label{sec:kernelcomp}
Putting all the above observations together means one can save a significant amount of computation time by the use of a recurrence when simultaneously using an operator valued polynomial approximation for the kernel $K(\mathrm{J}_x,\mathrm{J}_y)$ and then using the known commutation relations in (\ref{eq:commutation1}--\ref{eq:commutation2}). To illustrate the idea behind this approach we first discuss how to do this for a monomial kernel (or equivalently a kernel approximated in a monomial basis) and then show how these ideas can be expanded to arbitrary polynomial bases for the kernel using a variant of Clenshaw's algorithm.\\
Assuming a monomial expansion for the kernel, i.e. $K(x,y)= \sum_{n=0}^\infty \sum_{j=0}^n k_{nj} x^{n-j} y^j$, the primary part of the Volterra integration operator has the form
\begin{equation*}
\mathrm{Q}_y K(\mathrm{J}_x,\mathrm{J}_y) \mathrm{E}_y = \mathrm{Q}_y \left( \sum_{n=0}^\infty \sum_{j=0}^n k_{nj} \mathrm{J}_x^{n-j} \mathrm{J}_y^j \right) \mathrm{E}_y = \sum_{n=0}^\infty \sum_{j=0}^n k_{nj} \mathrm{J}^{n-j} \mathrm{Q}_y \mathrm{E}_y \mathrm{J}^j ,
\end{equation*}
where we have used the commutation relations in (\ref{eq:commutation1}--\ref{eq:commutation2}) to rewrite the summation using the Jacobi operator for the interval Jacobi polynomials.
Recalling that $\mathrm{Q}_y \mathrm{E}_y$ is a diagonal matrix which can be generated without any need to separately compute and multiply $\mathrm{Q}_y$ and $\mathrm{E}_y$, all that is left to compute are the required combinations of $\mathrm{Q}_y \mathrm{E}_y$ with the Jacobi operators, which can be built up recursively. This kind of recursive computation of all the required elements for the kernel can save significant computation cost if executed correctly. Since only the coefficients of $K(x,y)$ for this basis actually change across different problems one can in principle also store the basis elements $\mathrm{J}^{n-j} \mathrm{Q}_y \mathrm{E}_y \mathrm{J}^j$ and re-use them making this numerical evaluation of Volterra integrals even faster upon repeated use. This approach differs slightly depending on whether one intends to compute integrals from $0$ to $1-x$ or to compute integrals from $0$ to $x$. In the case of integrals from $0$ to $x$, one is either required to supply $K(1-x,y)$ to the algorithm or alternatively the Jacobi operators on the left can be replaced by $(\mathbb{1}-\mathrm{J})$ to account for the reflection, meaning that the basis elements become $(\mathbb{1}-\mathrm{J})^{n-j} \mathrm{Q}_y \mathrm{E}_y \mathrm{J}^j$. Taking the weight $\mathrm{W}_\mathrm{Q}$ into consideration the full Volterra integral operator is then
\begin{equation*}
    \mathrm{R} (\mathbb{1}-\mathrm{J}) \mathrm{Q}_y K(\mathrm{J}_x,\mathrm{J}_y) \mathrm{E}_y = \mathrm{R} (\mathbb{1}-\mathrm{J}) \sum_{n=0}^\infty \sum_{j=0}^n k_{nj} (\mathbb{1}-\mathrm{J})^{n-j} \mathrm{Q}_y \mathrm{E}_y \mathrm{J}^j
    .
\end{equation*}

This straightforward approach evidently only works if the kernel is of a form that may sensibly be approximated using monomials but it inspires an analogous approach based on expanding the kernel in its own orthogonal polynomial basis which need not be the same as those used to expand the function $f$. We use a variant of the Clenshaw algorithm introduced in section \ref{sec:clenshawintro} to build the kernel in terms of the Jacobi operators. In principle one could compute $K(\mathrm{J}_x,\mathrm{J}_y)$ as a full multiplication operator acting on a triangle Jacobi coefficient vector using an operator-valued version of Clenshaw's algoritm as discussed in \cite{olver_sparse_2019}. This is not the most efficient way to approach this problem, however, as it would mean losing the diagonal $\mathrm{Q}_y \mathrm{E}_y$ since for such an operator the multiplication with $K(\mathrm{J}_x,\mathrm{J}_y)$ would need to happen between $\mathrm{Q}_y$ and $\mathrm{E}_y$. Nevertheless, we will briefly discuss how to generate this multiplication by $K(\mathrm{J}_x,\mathrm{J}_y)$ operator in order to see which modifications one can make to this approach in order to respect the symmetries of the triangle and end up with recursive basis generation similar to the monomial kernel expansion case. 

The multiplication by $K(x,y)$ operator, which we label $\mathrm{M}_K$, can be written in an operator Clenshaw approach as (see \cite{olver_sparse_2019,olver_fast_2013,vasil_tensor_2016}):
\begin{equation}\label{eq:clenshawoperatormul}
\mathrm{M}_K = (\mathbf{e}_0 \otimes \mathbb{1})\mathcal{L}^\mathsf{-T} \mathbf{K}_\vartriangle,
\end{equation}
where $\otimes$ denotes the Kronecker product and $\mathcal{L}$ is defined as
\begin{equation*}
\mathcal{L} = \begin{pmatrix} (\mathbb{1}_1 \otimes \mathbb{1}) &  & & &  \\ (A_0^x \otimes \mathbb{1})-(\mathbb{1}_1 \otimes \mathrm{J}_x) & (B_0^x \otimes \mathbb{1}) & &  \\ (A_0^y \otimes \mathbb{1})-(\mathbb{1}_1 \otimes \mathrm{J}_y) & (B_0^y \otimes \mathbb{1}) & &   \\ (C_0^x \otimes \mathbb{1}) & (A_1^x \otimes \mathbb{1})-(\mathbb{1}_2 \otimes \mathrm{J}_x) & (B_1^x \otimes \mathbb{1}) &  \\ (C_0^y \otimes \mathbb{1}) & (A_1^x \otimes \mathbb{1})-(\mathbb{1}_2 \otimes \mathrm{J}_y) & (B_1^y  \otimes \mathbb{1})&  \\&\ddots&\ddots&  \ddots \end{pmatrix}.
\end{equation*}
As discussed for the Clenshaw evaluation method in section \ref{sec:clenshawintro} this system requires preconditioning to become solvable via backward substitution. For this case the preconditioner is
\begin{equation*}
\begin{pmatrix} (\mathbb{1}_1 \otimes \mathbb{1}) & & &  \\  & (B_0^+ \otimes \mathbb{1}) &  & \\ & &  (B_1^+ \otimes \mathbb{1})& \\ && &\ddots \end{pmatrix}\mathcal{L} = \tilde{\mathcal{L}},
\end{equation*}
with the $B_n^+$ defined as in section \ref{sec:clenshawintro}. Using such an operator valued Clenshaw algorithm one can compute $\mathrm{M}_K$ and thus obtain $\mathrm{Q}_y K(\mathrm{J}_x,\mathrm{J}_y) \mathrm{E}_y$ via $\mathrm{Q}_y \mathrm{M}_K \mathrm{E}_y$. However, as discussed above, for our purposes of Volterra integral operators this is computationally wasteful and misses the chance to take advantage of the triangle symmetries which allow for $\mathrm{Q}_y \mathrm{E}_y$ to be directly computable and diagonal. So instead we replace the $\mathbf{K}_\vartriangle$ in (\ref{eq:clenshawoperatormul}) by $(\mathbf{K}_\vartriangle \otimes \mathrm{Q}_y \mathrm{E}_y)$. The relations (\ref{eq:commutation1}--\ref{eq:commutation2}) then imply that all $\mathrm{J}_x$ operators may be replaced by a left multiplication with $\mathrm{J}$ and all $\mathrm{J}_y$ operators may be replaced by a right multiplication with $\mathrm{J}$ (respectively denoted by a $\diamond$ on the appropriate side). The system to solve thus becomes
\begin{equation*}
\mathrm{Q}_y K(\mathrm{J}_x,\mathrm{J}_y) \mathrm{E}_y = (\mathbf{e}_0 \otimes \mathbb{1})\mathcal{L}_V^\mathsf{-T} (\mathbf{K}_\vartriangle \otimes \mathrm{Q}_y \mathrm{E}_y),
\end{equation*}
with
\begin{equation*}
\mathcal{L}_V = \begin{pmatrix} (\mathbb{1}_1 \otimes \mathbb{1}) &  & & &  \\ (A_0^x \otimes \mathbb{1})-(\mathbb{1}_1 \otimes \mathrm{J} \diamond ) & (B_0^x \otimes \mathbb{1}) & &  \\ (A_0^y \otimes \mathbb{1})-(\mathbb{1}_1 \otimes \diamond\mathrm{J}) & (B_0^y \otimes \mathbb{1}) & &   \\ (C_0^x \otimes \mathbb{1}) & (A_1^x \otimes \mathbb{1})-(\mathbb{1}_2 \otimes  \mathrm{J} \diamond) & (B_1^x \otimes \mathbb{1}) &  \\ (C_0^y \otimes \mathbb{1}) & (A_1^x \otimes \mathbb{1})-(\mathbb{1}_2 \otimes \diamond \mathrm{J}) & (B_1^y  \otimes \mathbb{1})&  \\&\ddots&\ddots&  \ddots \end{pmatrix}.
\end{equation*}
After preconditioning as above, this allows the recursive and efficient computation of $\mathrm{Q}_y K(\mathrm{J}_x,\mathrm{J}_y) \mathrm{E}_y$ via an operator valued Clenshaw type algorithm while at the same time taking advantage of the diagonal nature of $\mathrm{Q}_y \mathrm{E}_y$. As in the monomial case, this approach has to be modified when integrating from $0$ to $x$ instead of from $0$ to $1-x$. In the $0$ to $x$ case one needs to take the reflection into account, which ends up either replacing all the left multiplications with $\mathrm{J}$ by left multiplications with $(\mathbb{1}-\mathrm{J})$ for the same reasons as above, while the right multiplications corresponding to $y$ multiplication remain the same, or requiring that $K(1-x,y)$ be supplied to the algorithm. Finally, this operator still requires left multiplication with the basis dependent weight $\mathrm{W}_\mathrm{Q}$ to represent the full Volterra integral operator for this approach.
\subsection{Numerical solutions to linear Volterra integral equations} \label{sec:methodvolterraIE}
The above described computational method for Volterra integrals has a natural extension to solving Volterra integral equations, which we describe in this section. Most generally a Volterra integral equation is any equation in which the unknown appears at least once as the integrand of a Volterra integral as defined in (\ref{eq:volterraintegral}) above. One usually distinguishes between at least two types of Volterra integral equations which are labelled Volterra integral equations of the first and second kind respectively. The Volterra integral equation of the first kind we will be interested in takes the following form:
\begin{equation}\label{eq:firstkind}
     \int_0^{x} K(x,y) u(y) \mathrm{d}y = g(x),
\end{equation}
where $u(x)$ is the unknown function to be solved for, $K(x,y)$ is a given kernel and $g(x)$ is a given function. Volterra integral equations of the second kind we will be interested in take the following form:
\begin{equation}\label{eq:secondkind}
    u(x) -  \int_0^{x} K(x,y) u(y) \mathrm{d}y = g(x),
\end{equation}
where once again $u(x)$ is the unknown function and $K(x,y)$ and $g(x)$ are given. While this is not further explored in this paper, there are natural extensions of these methods for other linear Volterra-type integral equations such as the third kind equations discussed in \cite{allaei_existence_2015,allaei_collocation_2017,song_analysis_2019}. \\ 
Whenever we write $\mathrm{Q}_y K(\mathbb{1}-\mathrm{J}_x,\mathrm{J}_y) \mathrm{E}_y$ in the coming sections, we mean to imply that this operator is computed using the Clenshaw approach detailed in section \ref{sec:kernelcomp}.
\subsubsection{Equations of the first kind}
Extending the above methods for Volterra integrals to Volterra integral equations is straightforward, though one needs to be mindful of the appropriate reflections. Using the above notation conventions, one way to write the Volterra integral equation of the first kind is
\begin{align*}
   \tilde{\mathbf{P}}^{(1,0)}(x)^\mathsf{T} (\mathbb{1}-\mathrm{J}) \mathrm{Q}_y K(\mathbb{1}-\mathrm{J}_x,\mathrm{J}_y) \mathrm{E}_y \mathbf{u} = \tilde{\mathbf{P}}^{(1,0)}(x)^\mathsf{T} \bar{\mathbf{g}} ,\\
  \Rightarrow   \tilde{\mathbf{P}}^{(1,0)}(x)^\mathsf{T}\mathbf{u} =  \tilde{\mathbf{P}}^{(1,0)}(x)^\mathsf{T}\left((\mathbb{1}-\mathrm{J}) \mathrm{Q}_y K(\mathbb{1}-\mathrm{J}_x,\mathrm{J}_y) \mathrm{E}_y \right)^{-1} \bar{\mathbf{g}}.
\end{align*}
The notation $\bar{\mathbf{g}}$ is used to indicate that we are directly supplying the coefficients of the reflected $g(1-x)$ to save an unnecessary additional reflection step, as formally we are solving the equivalent
\begin{equation}
     \int_0^{1-t} K(1-t,y) u(y) \mathrm{d}y = g(1-t).
\end{equation}
All function coefficient vectors in this section are initially expanded in the $\tilde{\mathbf{P}}^{(1,0)}(x)$ basis. This method works in numerical experiments but deriving convergence properties for it proves to be difficult (as is usual for Volterra equations of the first kind). However, under the condition that we can expand the function $q(x) = \frac{g(1-x)}{1-x}$ instead of $g(1-x)$ in $\tilde{\mathbf{P}}^{(1,0)}(x)$, one can find convergence conditions (see section \ref{sec:analysis} for details). Note that solvability of the Volterra integral equation of the first kind implies that both $g$ and $q$ must vanish when the upper limit of integration vanishes. When using $\mathbf{q}$ to denote the coefficient vector of $q(x) = \frac{g(1-x)}{1-x}$ the method then becomes
\begin{align*}
   \tilde{\mathbf{P}}^{(1,0)}(x)^\mathsf{T} \mathrm{Q}_y K(\mathbb{1}-\mathrm{J}_x,\mathrm{J}_y) \mathrm{E}_y \mathbf{u} = \tilde{\mathbf{P}}^{(1,0)}(x)^\mathsf{T} \mathbf{q} ,\\
  \Rightarrow   \tilde{\mathbf{P}}^{(1,0)}(x)^\mathsf{T}\mathbf{u} =  \tilde{\mathbf{P}}^{(1,0)}(x)^\mathsf{T}\left(\mathrm{Q}_y K(\mathbb{1}-\mathrm{J}_x,\mathrm{J}_y) \mathrm{E}_y \right)^{-1} \mathbf{q}.
\end{align*}
meaning that solving this type of equation for $u(x)$ is as simple as computing the coefficient vectors and operators (see the respective sections above for efficient ways to do so) and then solving a banded system of linear equations.
\subsubsection{Equations of the second kind}
Using the above-introduced weighted lowering operator $\mathrm{L}_{(1,0)}^{(0,0)}$ which shifts to the $\tilde{\mathbf{P}}^{(0,0)}(x)$ basis while multiplying with $(1-x)$, reflecting the result and then using a raising operator $\mathrm{S}_{(0,0)}^{(1,0)}$ to return to the $\tilde{\mathbf{P}}^{(1,0)}(x)$ basis we can write Volterra integral equations of the second kind as
\begin{align*}
 \tilde{\mathbf{P}}^{(1,0)}(x)^\mathsf{T} \left(\mathbb{1} - \mathrm{S}_{(0,0)}^{(1,0)} \mathrm{R}  \mathrm{L}_{(1,0)}^{(0,0)} \mathrm{Q}_y K(\mathbb{1}-\mathrm{J}_x,\mathrm{J}_y) \mathrm{E}_y  \right) \mathbf{u} = \tilde{\mathbf{P}}^{(1,0)}(x)^\mathsf{T} \mathbf{g} ,\\
  \Rightarrow   \tilde{\mathbf{P}}^{(1,0)}(x)^\mathsf{T} \mathbf{u} = \tilde{\mathbf{P}}^{(1,0)}(x)^\mathsf{T} \left(\mathbb{1} - \mathrm{S}_{(0,0)}^{(1,0)} \mathrm{R}  \mathrm{L}_{(1,0)}^{(0,0)} \mathrm{Q}_y K(\mathbb{1}-\mathrm{J}_x,\mathrm{J}_y) \mathrm{E}_y  \right)^{-1}\mathbf{g},
\end{align*}
which can once again be solved for $u(x)$ using any linear system of equations solver. Reflecting without the lowering and raising operator is not possible (although there are alternative ways to use such operators to accomplish the same goal) as this would result in an inconsistency between the bases used for the two appearances of $\mathbf{u}$.
\subsubsection{Different limits of integration}
As mentioned above, a similar derivation leads to an analogous method for Volterra integral equations of the first and second kind with different limits of integration:
\begin{align}\label{eq:canonicallimits}
    \int_0^{1-x} K(x,y) u(y) \mathrm{d}y = g(x),\\
    u(x) - \int_0^{1-x} K(x,y) u(y) \mathrm{d}y = g(x),
\end{align}
This results in an identity operator replacing the reflection and conversion operators in the above solution methods and in fact makes these types of equations even more efficient to solve but limits of integration of this sort are seen less often in applications. In particular, the operator version of Volterra integral equations of the first kind with limits of integration $0$ to $1-x$ is:
\begin{align*}
   \tilde{\mathbf{P}}^{(1,0)}(x)^\mathsf{T} \mathrm{Q}_y K(\mathrm{J}_x,\mathrm{J}_y) \mathrm{E}_y \mathbf{u} = \tilde{\mathbf{P}}^{(1,0)}(x)^\mathsf{T} \mathbf{q} ,\\
  \Rightarrow   \tilde{\mathbf{P}}^{(1,0)}(x)^\mathsf{T}\mathbf{u} =  \tilde{\mathbf{P}}^{(1,0)}(x)^\mathsf{T}\left(\mathrm{Q}_y K(\mathrm{J}_x,\mathrm{J}_y) \mathrm{E}_y \right)^{-1} \mathbf{q}.
\end{align*}
where now $\mathbf{q}$ is the coefficient vector of $q(x) = \frac{g(x)}{1-x}$. Equations of the second kind with these limits of integration can be written as:
\begin{align*}
 \tilde{\mathbf{P}}^{(1,0)}(x)^\mathsf{T} \left(\mathbb{1} - (\mathbb{1}-\mathrm{J}) \mathrm{Q}_y K(\mathrm{J}_x,\mathrm{J}_y) \mathrm{E}_y  \right) \mathbf{u} = \tilde{\mathbf{P}}^{(1,0)}(x)^\mathsf{T} \mathbf{g} ,\\
  \Rightarrow   \tilde{\mathbf{P}}^{(1,0)}(x)^\mathsf{T} \mathbf{u} = \tilde{\mathbf{P}}^{(1,0)}(x)^\mathsf{T} \left(\mathbb{1} - (\mathbb{1}-\mathrm{J}) \mathrm{Q}_y K(\mathrm{J}_x,\mathrm{J}_y) \mathrm{E}_y  \right)^{-1}\mathbf{g}.
\end{align*}
We present an implementation of both options for the limits of integration in the next section.
\section{Examples and applications} \label{sec:numericalexamples}
We present three sets of numerical examples to validate our implementation. The first set concerns itself with Volterra integral equations of the first kind, the second with Volterra integral equations of the second kind with kernels of varying oscillatory intensity and the third set discusses a singular Volterra integral equation stemming from a heat conduction problem with mixed boundary conditions. As oscillatory kernels require high orders of polynomials to approximate accurately and the method was not designed for singular kernels, the second and third set are designed to test the method's stability.\\
The computations presented in this section have been performed with an implementation of the scheme in the Julia programming language~\cite{beks2017} in the framework of ApproxFun.jl and MultivariateOrthogonalPolynomials.jl \cite{olver_practical_2014,olver_fast_2013,townsend_automatic_2015}. The coefficients of the solution have relative accuracy with standard floating point arithmetic, even as they decay below machine precision.  Values for absolute errors presented in this section converge beyond the precision of 64-bit floating point numbers because of the rapid convergence of the method and the way ApproxFun.jl implements function approximation (cf. \cite{olver_practical_2014,olver_fast_2013,townsend_automatic_2015})---the only time beyond 64-bit floating point precision numbers (via "BigFloat") were used is in the analytic solutions used as comparisons, as otherwise the convergence of the error would be capped by the precision at which the analytic solution is evaluated.
\subsection{Set 1: Volterra integral equations of the first kind}
We investigate the numerical solution of the following two example Volterra integral equations of the first kind:
\begin{equation}\label{eq:set1_analyticintegral}
    e^{-x} + e^x (-1 + 2 x) = 4 \int_0^{x} e^{y-x} u_1(y) \mathrm{d}y.
\end{equation}
\begin{equation}\label{eq:set1_involvedintegral}
    \frac{\mathrm{sin}(4 \pi^2 x^2)}{x} = \int_0^{x} e^{-10\left(x-\frac{1}{3}\right)^2 - 10\left(y - \frac{1}{3}\right)^2} u_2(y) \mathrm{d}y.
\end{equation}
The analytic solution to the first equation can be found to be:
\begin{equation*}
    u_1(x) = xe^x.
\end{equation*}
We present the absolute error between the analytic and numerical solution for $u_1(x)$ using the orthogonal polynomial method introduced in this paper in Figure \ref{fig:set12_errors}A for different matrix dimensions $n \times n$ and the absolute error between the numerical solution for $u_2(x)$ and a high degree solution computed with $n=5050$ in Figure \ref{fig:set12_errors}B.
\subsection{Set 2: Volterra integral equations of the second kind with oscillatory kernels}
We seek numerical solutions $u_1$, $u_2$ and $u_3$ to the following three Volterra integral equations of the second kind with kernels of varying oscillatory intensity:
\begin{align}
\label{eq:set2K1}  u_1(x) &= \tfrac{e^{-10\pi x}(1+20 \pi)-2+\mathrm{cos}(10\pi x)+\mathrm{sin}(10\pi x)}{20\pi} + \int_0^{x} \left(1-\mathrm{cos}\left( 10 \pi x-10 \pi y \right)\right) u_1(y) \mathrm{d}y\\ 
 \label{eq:set2K2}   u_2(x) &= \frac{e^\frac{x}{2}}{\pi} + \int_0^{x} \left( \mathrm{sin}(10\pi x)+\mathrm{cos}(10 \pi y) \right) u_2(y) \mathrm{d}y\\ 
    u_3(x) &= e^{x^2-2x} + \int_0^{1-x} \left( -2x + y + \mathrm{sin}(25x^2 + 8 \pi y) \right) u_3(y) \mathrm{d}y. \label{eq:set2K3}
\end{align}
Accurate approximation of these kernels on the canonical triangle domain requires coefficient vectors of length exceeding $10^3$. We include contour plots of the specified kernels on said domain in Figure \ref{fig:set2_contours}. One can find an analytic solution to the first equation:
\begin{equation*}
u_1(x) = e^{-10 \pi x}.
\end{equation*}
For the other two equations, we instead compare to a numerical solution of high degree ($n=5050$). We plot the absolute error convergence of the numerical solutions in Figure \ref{fig:set2_errors}. Due to the oscillatory character of these kernels and the number of coefficients involved, this can be considered a moderate stress test of the Clenshaw approach to the computations of the Volterra integral operator.
\subsection{Set 3: Singular Volterra integral equation of the second kind in heat conduction with mixed boundary conditions}
Finally we discuss a more application-oriented example discussed in a handful of different variations in \cite{diogo_hermite-type_1991,diogo_high_2004,diogo_numerical_2006,wazwaz_two_2016,baratella_nystrom_2009}:
\begin{equation}\label{eq:conductionexample}
    u(x) = g(x) + \int_0^{x} \frac{y^{\mu-1}}{x^\mu} u(y) \mathrm{d}y.
\end{equation}
To see how equations of this type can result from heat conduction problems of the form $\frac{\partial^2 u}{\partial x} - \frac{1}{\alpha^2} \frac{\partial u}{\partial y} = 0$ with mixed boundary conditions, see for example \cite{diogo_high_2004}. This equation varies both in its singularity properties as well as its number of solutions depending on the parameter $\mu$. This example equation stemming from an application of Volterra integrals demonstrates that the method developed in this paper has a broader range of applicability and can in some cases extend to certain classes of singular problems as well, despite this not being part of the considerations during the development of the method. For testing purposes we choose the following for $g(x)$:
\begin{align*}
    g_1(x) &= (1+x+x^2)\\
    g_2(x) &= \frac{(1+4 \pi^2 x^2) \mathrm{sinh}(2\pi x)-2\pi x \mathrm{cosh}(2\pi x)}{4 \pi^2 x^2}.
\end{align*}
The following analytic solutions to these equations can be found for general $\mu$ for $g_1$ (e.g. in \cite{wazwaz_two_2016}) and for $\mu=3$ for $g_2$:
\begin{align*}
    u_1(x,\mu) &= \frac{\mu}{\mu-1}+\frac{\mu+1}{\mu}x+\frac{\mu+2}{\mu+1}x^2,\\
    u_2(x,\mu=3) &= \mathrm{sinh}(2 \pi x).
\end{align*}
As the kernel is separable, the problem can instead be treated as
$$ x^\mu u(x) = x^\mu g(x) + \int_0^{x} y^{\mu-1} u(y) \mathrm{d}y,$$
which can be solved by appropriately adding multiplications with Jacobi operators or altering the supplied $g(x)$ in the method to solve Volterra integral equations of the second kind. We plot numerical solutions obtained for $g_1(x)$ with $\mu=7$ and $g_2(x)$ with $\mu=3$ in Figure \ref{fig:set3}. The naturally more error prone neighborhood of the singularity can be well approximated arbitrarily close to the singularity (though not at the exact point of the singularity itself) using higher values of $n$ if needed. For $g_2(x)$ the method shows no instability at the weak singularity of the kernel.
\begin{figure}[H]
    \centering
     \subfloat[]
    {{  \includegraphics[width=7cm]{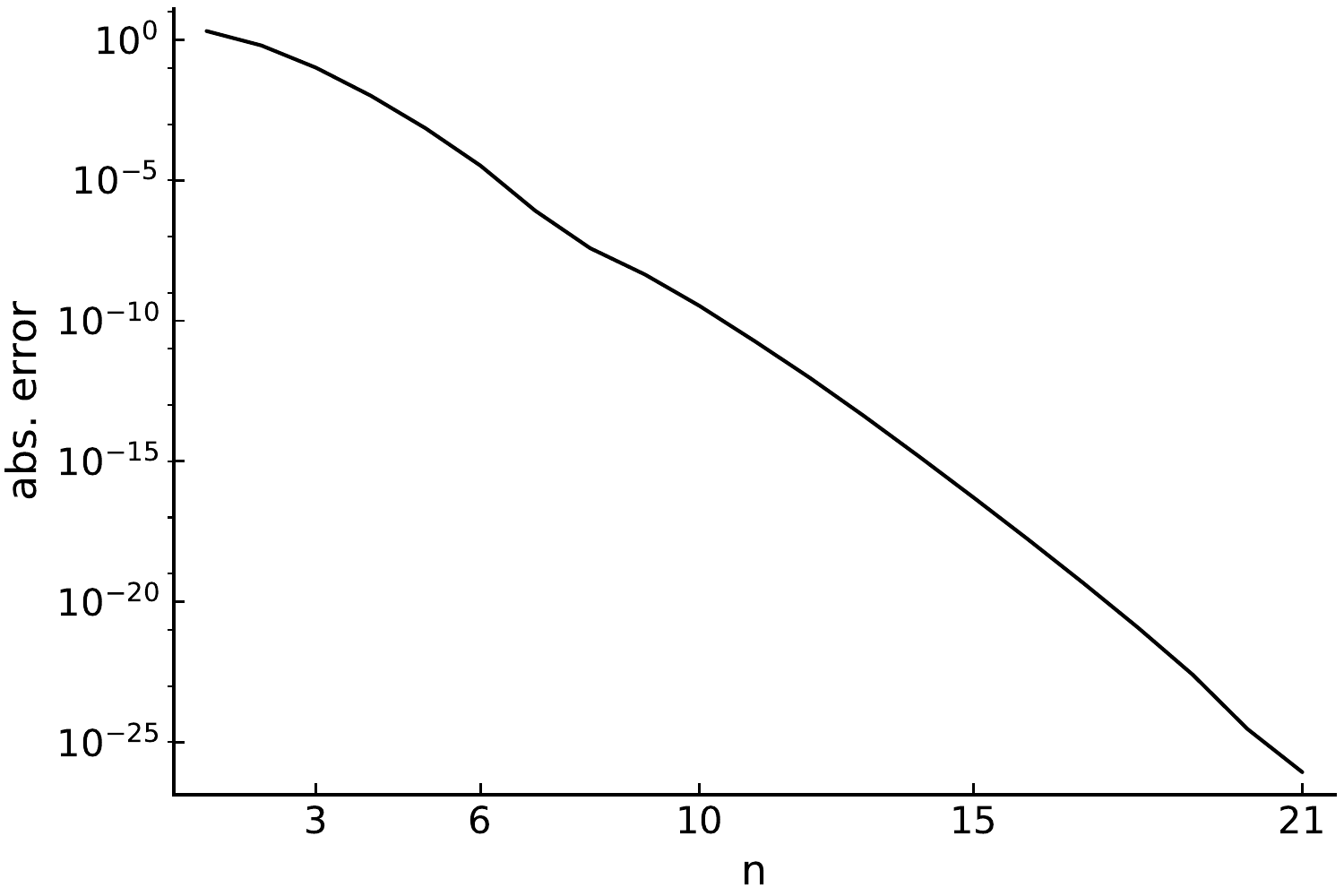} }}
       \qquad
     \subfloat[]
    {{  \includegraphics[width=7cm]{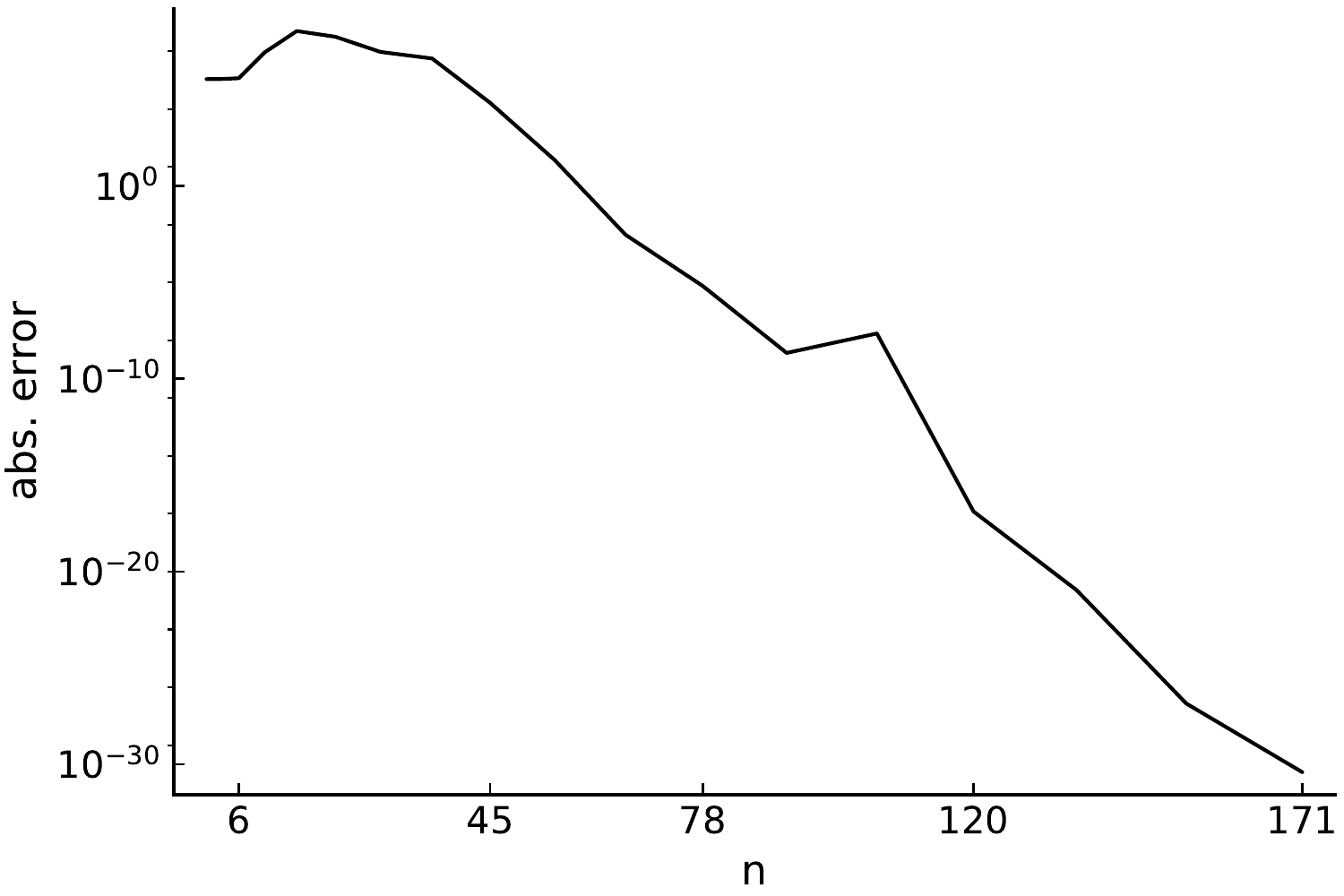} }}
           \qquad
    \caption{(A) shows absolute error between (\ref{eq:set1_analyticintegral}) and the known analytic solution while (B) compares (\ref{eq:set1_involvedintegral}) to a solution computed with $n=5050$. }%
    \label{fig:set12_errors}%
\end{figure}
\begin{figure}[H]
    \centering
     \subfloat[$K_1(x,y)$]
    {{  \includegraphics[width=4.2cm]{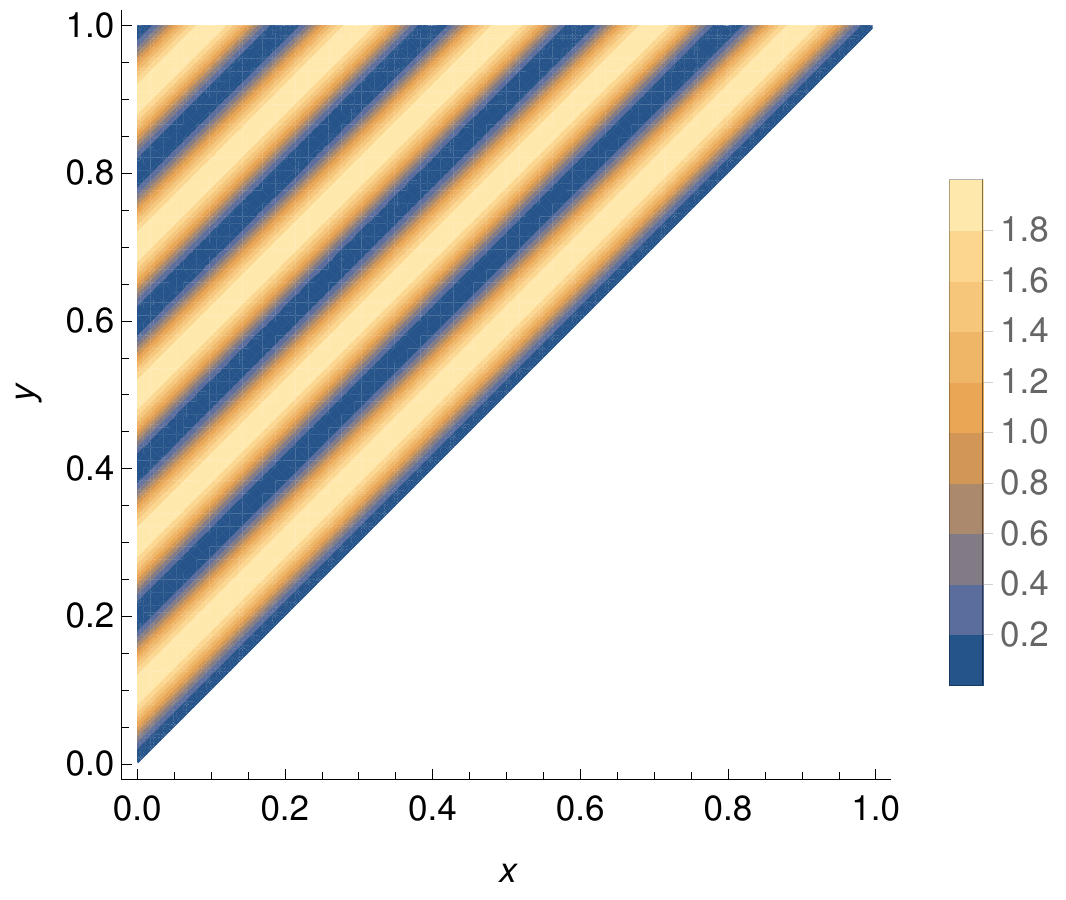} }}
       \qquad
     \subfloat[$K_2(x,y)$]
    {{  \includegraphics[width=4.2cm]{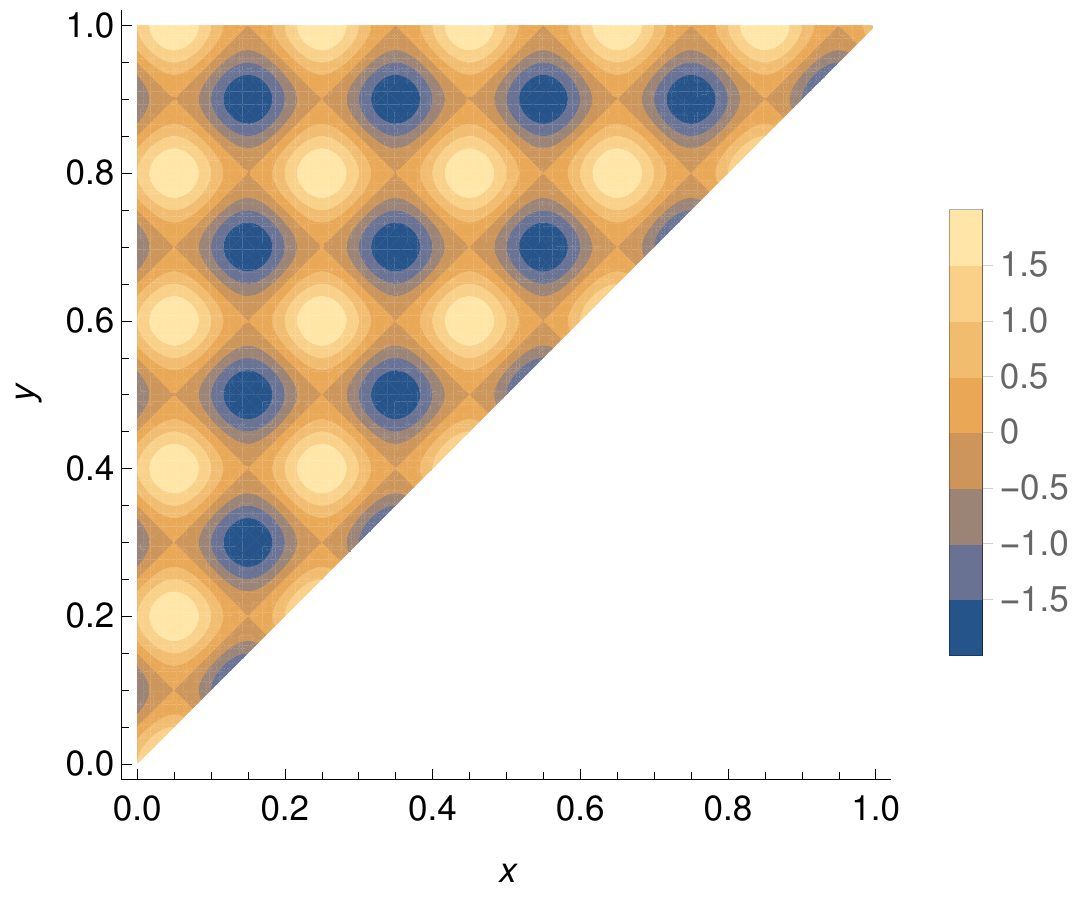} }}
           \qquad
    \subfloat[$K_3(x,y)$]
    {{  \includegraphics[width=4.2cm]{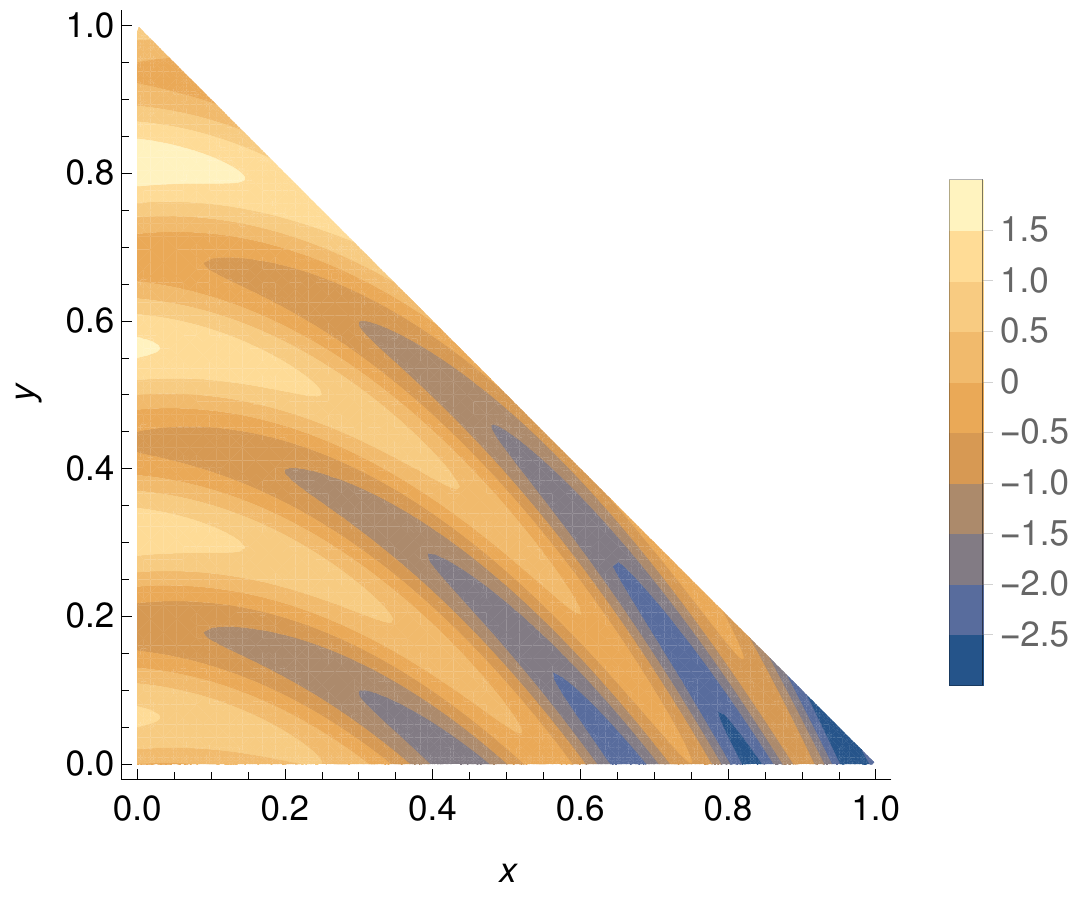} }}
           \qquad
    \caption{Contour plots of oscillatory kernels for equations (\ref{eq:set2K1}--\ref{eq:set2K3}) on their natural triangle domains.}%
    \label{fig:set2_contours}%
\end{figure}
\begin{figure}[H]
    \centering
     \subfloat[]
    {{  \includegraphics[width=7cm]{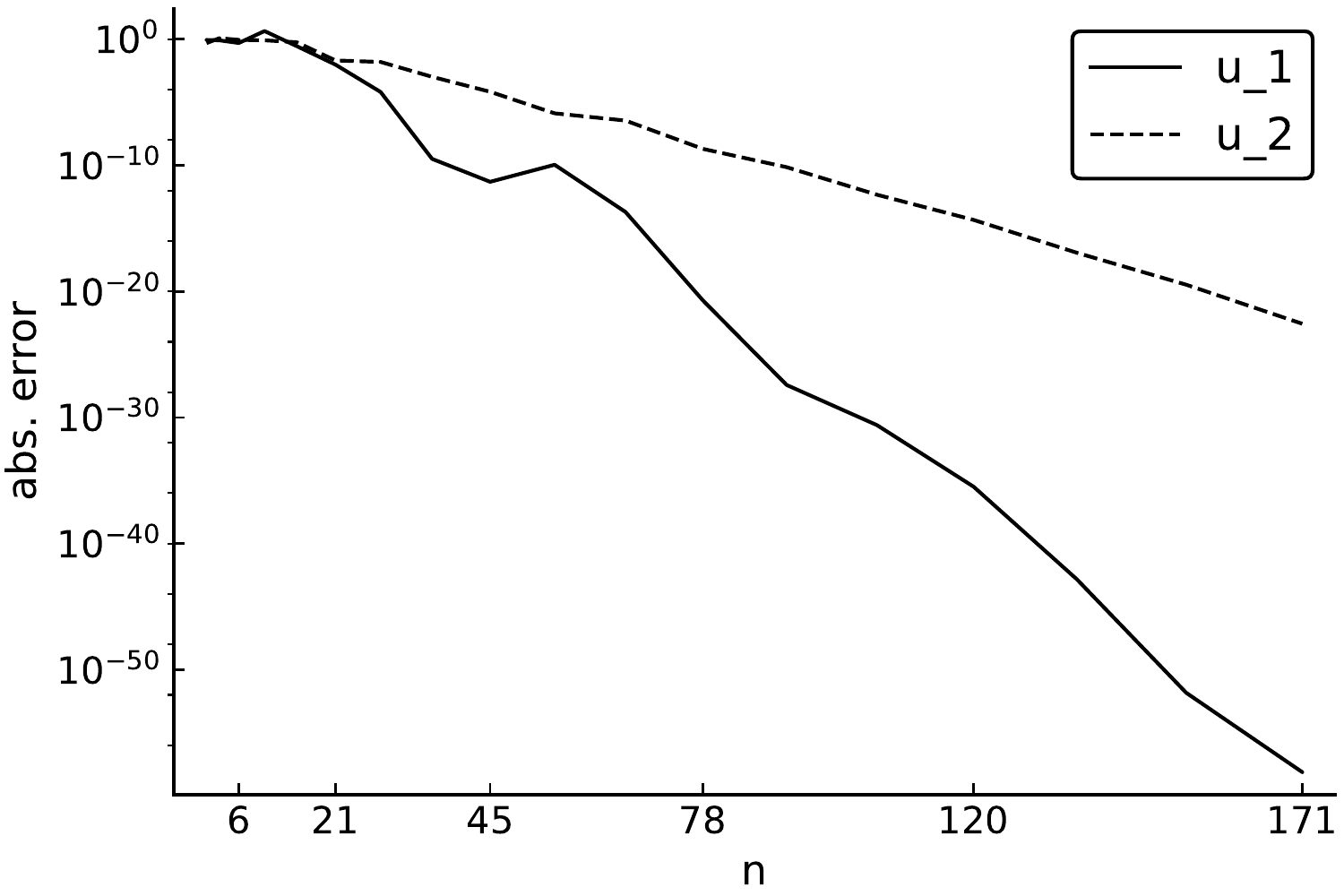} }}
           \qquad
    \subfloat[]
    {{  \includegraphics[width=7cm]{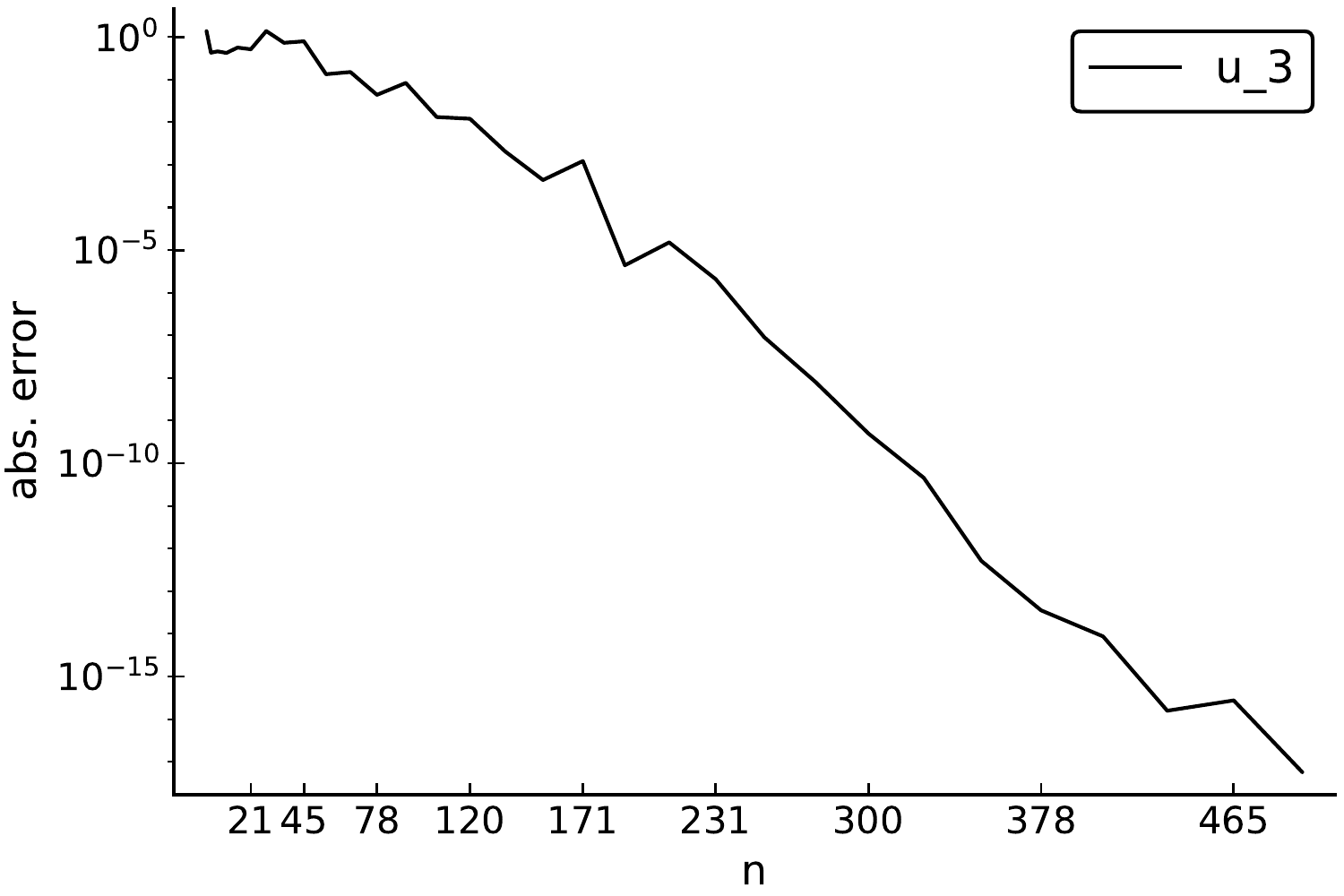} }}
           \qquad
    \caption{Absolute errors for equations (\ref{eq:set2K1}--\ref{eq:set2K3}). $u_1(x)$ is compared to the analytic solution, $u_2(x)$ and $u_3(x)$ are compared to a solution computed with $n=5050$.}%
    \label{fig:set2_errors}%
\end{figure}
\begin{figure}[H]
    \centering
     \subfloat[$g_1(x)$ with $\mu=7$]
    {{  \includegraphics[width=7cm]{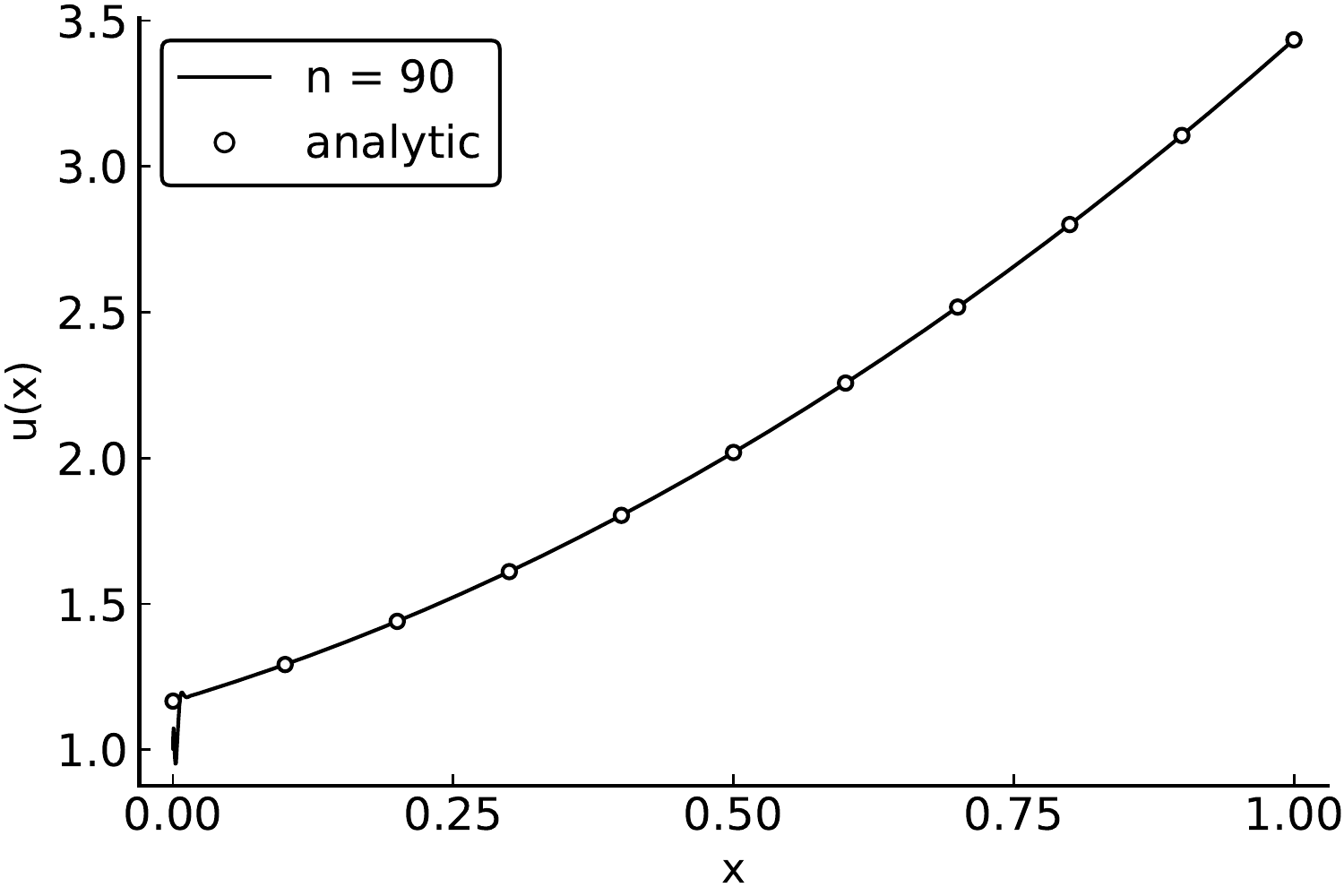} }}
       \qquad
     \subfloat[$g_2(x)$ with $\mu=3$]
    {{  \includegraphics[width=7cm]{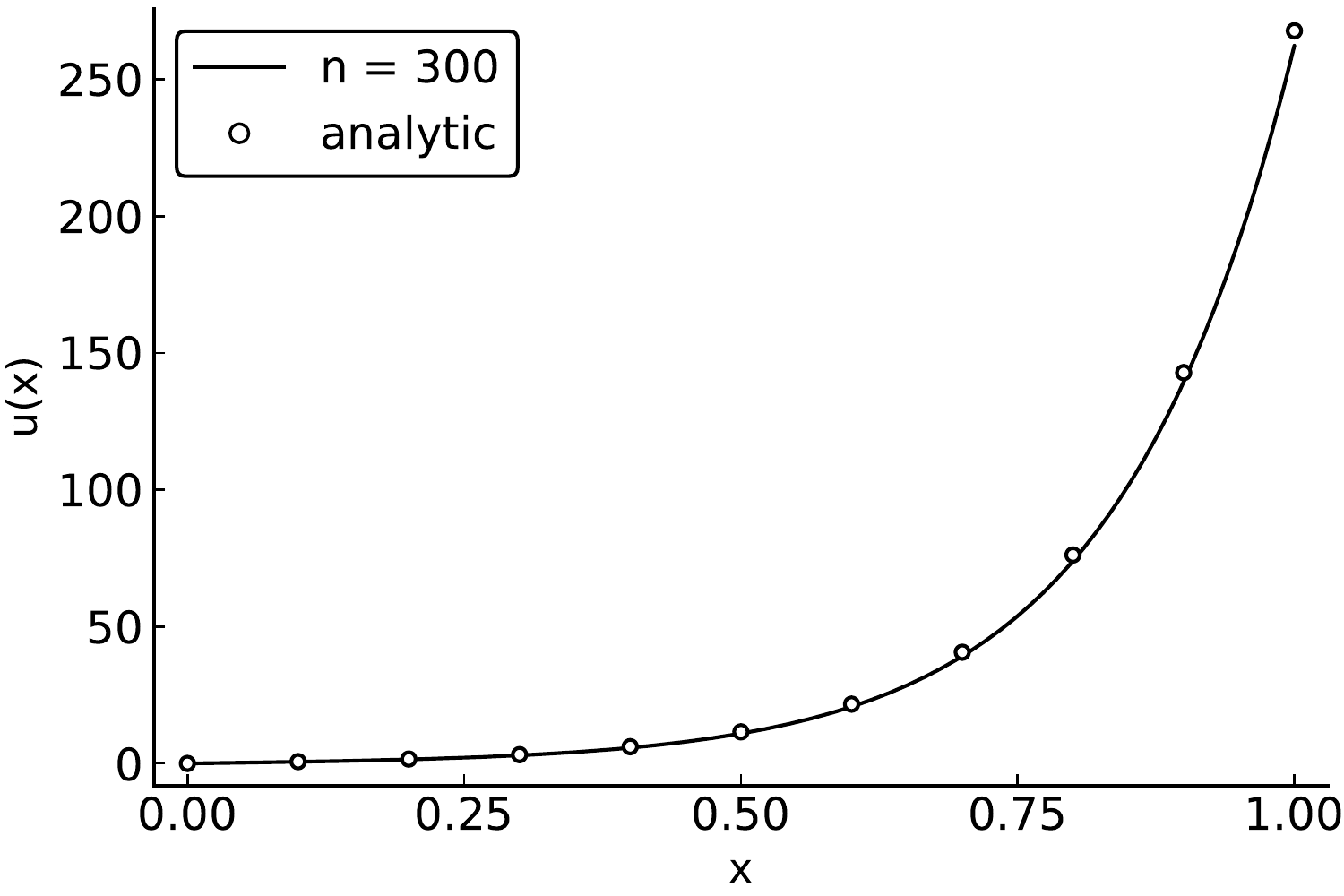} }}
           \qquad
    \caption{Numerical and analytic solutions to the problem in (\ref{eq:conductionexample}).}%
    \label{fig:set3}%
\end{figure}

\section{Stability and convergence of the method} \label{sec:analysis}
In this section we make use of the fact that the coefficient space of orthogonal polynomials is equivalent to an infinite-dimensional Banach space (in particular a sequence space). The strategy for the  analysis of the method is to show that the operators to be inverted for Volterra integral equations of the second kind can be written as compact perturbations of the identity (compare \cite{olver_fast_2013,slevinsky_singular_2017,lintner_generalized_2015}), i.e. can be written as
\begin{equation}  \label{eq:compactnesscriterion}
( \mathbb{1} + \mathcal{K}) u = g
\end{equation}
where $\mathcal{K}$ is compact. Operators of this form are either invertible or neither injective nor surjective by the Fredholm alternative, cf. \cite{bachman_functional_2000,lindenstrauss_classical_1996}. The assumption of well-posedness for the equation thus guarantees that an operator of this form is invertible and standard convergence results for finite section methods \cite{bottcher_analysis_2006} then guarantee convergence. We begin by discussing the solver for Volterra integral equations of the second kind, as the analysis for first kind problems is more involved.\\

\subsection{Equations of the second kind}
\begin{definition} 
We define the projection operators $\mathcal{P}_n: \ell^2 \rightarrow \ell^2$ which map a given coefficient vector to a truncated version of itself with non-zero entries for the first $n$ coefficients only.
\end{definition}

\begin{definition}
The {\it analysis operator} $\mathcal{E} : L^2(0,1) \rightarrow \ell^2$ is the inclusion of a square integrable function into the $\ell^2$ coefficient space of the complete basis of orthogonal Jacobi polynomials, which is guaranteed to exist by the Stone–Weierstrass theorem and is a bounded operator. The {\it synthesis operator} is its inverse $\mathcal{E}^{-1} : \ell^2 \rightarrow L^2(0,1)$, which is also bounded.
Note the terms analysis and synthesis are  terminology in frame theory  \cite{casazza_finite_2013,christensen_introduction_2003}.
\end{definition}

\begin{lemma}[]The coefficient space Volterra integral operator $\mathrm{V}_K$ is compact, where $\mathrm{V}_K: \ell^2 \rightarrow \ell^2$ for a given kernel $K(x,y) \in L^2[T^2]$ with limits of integration $0$ to $x$ acting on the coefficient vector Banach space $\ell^2$ of the Jacobi polynomials $\tilde{\mathbf{P}}^{(1,0)}(x)$ is of the form $$\mathrm{V}_K = \mathrm{L}_{(1,0)}^{(0,0)} \mathrm{Q}_y K(\mathbb{1}-\mathrm{J}_x,\mathrm{J}_y) \mathrm{E}_y,$$ with the respective operators defined as in section $\ref{sec:volterra}$.
\end{lemma}

\begin{proof}
$\mathrm{V}_K = \mathrm{L}_{(1,0)}^{(0,0)} \mathrm{Q}_y K(\mathbb{1}-\mathrm{J}_x,\mathrm{J}_y) \mathrm{E}_y$ follows from the definition of the involved operators, see section \ref{sec:volterra}. To see compactness of $\mathrm{V}_K$ we consider the following diagram of functions between Banach spaces which represents the formalized version of the method:
\begin{figure}[H]
\centering
\begin{tikzpicture}
  \matrix (m) [matrix of math nodes,row sep=3em,column sep=4em,minimum width=2em]
  {
     L^2(0,1) & L^2(0,1) \\
     \ell^2 & \ell^2 \\};
  \path[-stealth]
    (m-1-1) edge node [left] {$\mathcal{E}$} (m-2-1)
            edge node [above] {$\mathcal{V}_K$} (m-1-2)
    (m-2-1.east|-m-2-2) edge node [below] {$\mathrm{V}_K$}
            node [above] {} (m-2-2)
     (m-2-2) edge node [right] {$\mathcal{E}^{-1}$} (m-1-2);
\end{tikzpicture}
\end{figure}
\noindent $\mathcal{V}_K$ for a kernel $K(x,y) \in L^2[T^2]$ is the Volterra integral operator for said kernel acting on $L^2(0,1)$. It is a classical result of functional analysis that such Volterra integral operators $\mathcal{V}_K$ are Hilbert–Schmidt operators and thus compact \cite{muscat_functional_2014}. It follows that $\mathrm{V}_K = \mathcal{E} \circ \mathcal{V}_K \circ \mathcal{E}^{-1}$ is a finite composition of bounded and compact operators between Banach spaces and hence itself compact.
%
%
\end{proof}

\begin{lemma}[] For $\mathrm{V}_K$ and $\mathcal{P}_n$ defined as above, we have   $$\lim_{n\rightarrow\infty}\|\mathrm{V}_K-\mathcal{P}_n\mathrm{V}_K\mathcal{P}_n^\mathsf{T}\| = 0.$$ \end{lemma}

\begin{proof}
 This follows directly from the compactness of $\mathrm{V}_K$ and the fact that $\ell^2$ is a Hilbert space and thus has the approximation property \cite{lindenstrauss_classical_1996}.  \end{proof}
 
The above lemma justifies referring to the finite-dimensional projections $\mathcal{P}_n\mathrm{V}_K\mathcal{P}_n^\mathsf{T}$ of the Volterra operator as approximations. 
 
\begin{lemma}[]$\mathrm{S}_{(0,0)}^{(1,0)} \mathrm{R}  \mathrm{L}_{(1,0)}^{(0,0)} \mathrm{Q}_y K(\mathbb{1}-\mathrm{J}_x,\mathrm{J}_y) \mathrm{E}_y$ is compact on $\ell^2$ and thus Volterra integral equations of the second kind can be written in the form $(\mathbb{1} + \mathcal{K}) \mathbf{u} = \mathbf{g}$ with $\mathcal{K}$ compact. \end{lemma}

\begin{proof}
The operators $\mathrm{S}_{(0,0)}^{(1,0)}$ and $\mathrm{R}$ acting on the Banach space $\ell^2$ can both readily be seen to be bounded operators from their definitions  from the Jacobi polynomial's recurrence relationships \cite[18.9.5]{nist_2018}. The result then follows from the observation that the Volterra integral operator $\mathrm{L}_{(1,0)}^{(0,0)} \mathrm{Q}_y K(\mathbb{1}-\mathrm{J}_x,\mathrm{J}_y) \mathrm{E}_y$ was shown to be compact and composition of bounded operators with a compact operator yields a compact operator.
\end{proof}

An analogous chain of arguments immediately establishes:

\begin{lemma}[] The Volterra integral operator for the limits $0$ to $1-x$ is compact and can be written as $$\mathrm{V}_K = (\mathbb{1}-\mathrm{J}) \mathrm{Q}_y K(\mathrm{J}_x,\mathrm{J}_y) \mathrm{E}_y.$$ The method is thus also of the form in  (\ref{eq:compactnesscriterion}).
\end{lemma}

\begin{corollary}[] The method described in section \ref{sec:methodvolterraIE} converges like $\| \mathbf{u} - \mathcal{P}_n \mathbf{u} \| \rightarrow 0$  as $n\rightarrow \infty$ for well-posed Volterra integral equations of the second kind. \end{corollary}
\begin{proof}
As the method is of the form in (\ref{eq:compactnesscriterion}), i.e. $( \mathbb{1} + \mathcal{K}) \mathbf{u} = \mathbf{g}$ with $\mathcal{K}$ compact, the result is a corollary of the above results combined with the known invertibility and convergence properties for problems of this form in finite section methods, see e.g. \cite{bottcher_analysis_2006}.
\end{proof}

\subsection{Equations of the first kind}
The Fredholm alternative and Neumann series arguments underlying the proofs above break down for first kind problems as the Volterra operator $\mathrm{V}_K: \ell^2 \rightarrow \ell^2$ is compact  on the infinite dimensional Banach space $\ell^2$ and therefore is strictly singular, cf. \cite{bachman_functional_2000}. Thus, while the finite dimensional approximations $\mathrm{V}_n$ of the Volterra operator may have an inverse $\mathrm{V}_n^{-1}$, it is not obvious that $\mathbf{u}_n = \mathrm{V}_n^{-1} \mathbf{q}$ converges to $\mathbf{u}$ in the limit. The problem can be made well-posed, however, if one considers the Volterra operator as a map between two different appropriately chosen Banach spaces. Under sufficient continuity assumptions as well as the assumption that a given Volterra integral equation of the first kind has a solution, this problem may then be salvaged by finding a preconditioner which allows us to rewrite it as a problem involving operators which are compact perturbations of Toeplitz operators. We begin by assuming a polynomial kernel from where an extension argument directly yields that it also applies for the non-polynomial case. Note that in this section we will prove convergence of the method only for the case of limits of integration $0$ to $1-x$. This is not a limitation for the case of integral equations of the first kind, since solving
\begin{equation*}
     \int_0^{t} K(t,y) u(y) \mathrm{d}y = g(t).
\end{equation*}
and
\begin{equation*}
     \int_0^{1-x} K(1-x,y) u(y) \mathrm{d}y = g(1-x).
\end{equation*}
are formally equivalent, as solving one automatically solves the other with $t=1-x$. The reason for the particular choice for our proofs is that some arguments are more clear in this variant. Furthermore, as the monomial expansion and Clenshaw algorithm based Volterra operators are  exactly the same for polynomial kernels the analysis will make use of the simpler structure of the former. \\
To discuss invertibility for equations of the first kind we need to reframe the Volterra operator as a a map between two different Banach spaces, which are similar in spirit to Sobolev spaces.

\begin{definition} Let $\ell^2_{\lambda}$ with $\lambda \geq 0$ denote the Banach space with norm $$ \|\mathbf{u}\|_{\ell^2_\lambda} = \sqrt{\sum_{n=0}^\infty \left((1+n)^\lambda |u_n| \right)^2} < \infty.$$ \end{definition}

Any $\mathbf{u} \in \ell^2_\lambda$ corresponds uniquely to a $\mathbf{u} \in \ell^2$ so we have $\ell^2_\lambda \subset \ell^2$ whereas the converse is clearly not the case.

\begin{lemma}\label{theorem:firstkindfullvoltlemma} Let  $\mathrm{V}_K: \ell^2 \rightarrow \ell^2_1$ denote the Volterra operator in coefficient space of $\tilde{\mathbf{P}}^{(1,0)}(x)$ with limits of integration $0$ to $1-x$ for a given polynomial kernel 
$$
K(x,y)=\sum_{n=0}^M \sum_{j=0}^n k_{nj} x^{n-j} y^{j}.
$$
Then $$ \mathrm{V}_K = (\mathbb{1}-\mathrm{J})\mathrm{D}\left(\mathrm{D}^{-1} \sum_{n=0}^M \sum_{j=0}^n k_{nj} \mathrm{J}^{n-j} \mathrm{D} \mathrm{J}^j \right),$$ with $\mathrm{D} = \mathrm{Q}_y \mathrm{E}_y$, $\mathrm{D}: \ell^2 \rightarrow \ell^2_1$ and $\mathrm{D}^{-1}: \ell^2_1 \rightarrow \ell^2$.
\end{lemma}
\begin{proof}
That $\mathrm{D} = \mathrm{Q}_y \mathrm{E}_y$ is diagonal with entries $\frac{(-1)^{n+1}}{n}$ is due to properties of the Jacobi polynomials, see section \ref{sec:volterra} as well as \cite[18.6.1 and 18.17.1]{nist_2018}. The important observation to make is that  $\mathrm{D}$ can be thought of as $\mathrm{D}: \ell^2 \rightarrow \ell^2_1$, which makes $\mathrm{D}$ a bounded and invertible operator with $\mathrm{D}^{-1}: \ell^2_1 \rightarrow \ell^2$. With $\mathrm{V}_K$ and $K(x,y)$ as above, we thus have $$ \mathrm{V}_K = (\mathbb{1}-\mathrm{J}) \sum_{n=0}^M \sum_{j=0}^n \mathrm{J}^{n-j} \mathrm{D} \mathrm{J}^j = (\mathbb{1}-\mathrm{J}) \mathrm{D}\left(\mathrm{D}^{-1} \sum_{n=0}^M \sum_{j=0}^n k_{nj} \mathrm{J}^{n-j} \mathrm{D} \mathrm{J}^j \right),$$ via Section \ref{sec:kernelcomp}.
\end{proof}

\begin{definition} When solving Volterra integral equations of the first kind with the method described in Section \ref{sec:methodvolterraIE}, it is useful to distinguish the operator without the weight $(1-x)$ which is to be inverted from the full Volterra operator. We will denote this operator $\tilde{\mathrm{V}}_K: \ell^2 \rightarrow \ell^2_1$, where
$$ (\mathbb{1}-\mathrm{J}) \tilde{\mathrm{V}}_K  = \mathrm{V}_K.$$
We furthermore see that
 $$ \tilde{\mathrm{V}}_K = \mathrm{D}\left(\mathrm{D}^{-1} \sum_{n=0}^M \sum_{j=0}^n k_{nj} \mathrm{J}^{n-j} \mathrm{D} \mathrm{J}^j \right).$$
as an immediate corollary of Lemma \ref{theorem:firstkindfullvoltlemma}.
\end{definition}

\begin{lemma}
$\tilde{\mathrm{V}}_K$ may be written as $$\tilde{\mathrm{V}}_K = \mathrm{D}(\mathrm{T}[f]+\mathcal{K}),$$ where $\mathrm{T}[f]$ is a Toeplitz operator with symbol $f$ and $\mathcal{K}$ is compact. Furthermore, the symbol  is uniquely determined by the coefficients of the polynomial kernel $K(x,y)=\sum_{n=0}^M \sum_{j=0}^n k_{nj} x^{n-j} y^{j}$ to be 
$$ 
f(z) = \sum_{n=0}^M \sum_{j=0}^n k_{nj} \cos^{2n}\left(\frac{\theta}{2}\right)\qquad \hbox{where}\qquad z = {\rm e}^{{\rm i} \theta}.
$$
\end{lemma}
\begin{proof} 
From the Lemma \ref{theorem:firstkindfullvoltlemma} we see that the first statement is equivalent to the claim that
$$\sum_{n=0}^M \sum_{j=0}^n k_{nj} \mathrm{D}^{-1} \mathrm{J}^{n-j} \mathrm{D} \mathrm{J}^j$$ 
is of the form $\mathrm{T}+\mathcal{K}$ and thus asymptotically Toeplitz. To show this we need two observations: First, under sufficient continuity assumptions for the kernel, which are satisfied due to the kernel being polynomial, we have that
\begin{equation}\label{eq:contsymbols}
T[a]T[b] = T[ab] - H[a]H[\bar b],
\end{equation} 
and in particular
$$ T[a]T[a] = T[a^2] - H[a]H[\bar a],$$
where $H[a]$, $H[\bar a]$ and $H[\bar b]$ are compact Hankel operators \cite{bttcher_introduction_1999}.
Thus any asymptotically Toeplitz operator (of sufficiently continuous symbol) raised to a finite power is again an asymptotically Toeplitz operator, as $(T+\mathcal{K})^2 = T^2 + T \mathcal{K} + \mathcal{K}T + \mathcal{K}^2$ and $T^2$ is again Toeplitz plus something compact via the above relation. The composition of bounded operators with compact operators is compact making $T \mathcal{K} + \mathcal{K}T + \mathcal{K}^2$ compact. An induction argument demonstrates that this is true for any power $n\in \mathbb{N}$. In particular, since it is known that $\mathrm{J}$ is a compact perturbation of a Toeplitz operator \cite{nist_2018} we know that $\mathrm{J}^j$ is a compact perturbation of a Toeplitz operator as well. The second observation is that for the banded operator $\mathrm{J}^{n-j}$, the operator $\mathrm{D}^{-1}\mathrm{J}^{n-j} \mathrm{D} $ is also a compact perturbation of a Toeplitz operator and in fact we have that $\mathrm{J}^{n-j}$ and $\mathrm{D}^{-1}\mathrm{J}^{n-j} \mathrm{D}$ differ only in their compact part, i.e. have the same Toeplitz component. Via (\ref{eq:contsymbols}) we thus have that $\sum_{n=0}^M \sum_{j=0}^n k_{nj} \mathrm{D}^{-1} \mathrm{J}^{n-j} \mathrm{D} \mathrm{J}^j$ is of the form $(\mathrm{T}+\mathcal{K})$ and thus asymptotically Toeplitz.\\
Along with the above observations, Equation (\ref{eq:contsymbols}) tells us that we can compute the symbol of the Toeplitz part of a product of operators which are compact perturbations of Toeplitz operators if we know the symbols of the individual Toeplitz components. Due to bandedness it is straightforward to confirm that the symbol of the Toeplitz part of the multiplication operator $\mathrm{J}$ is $(\frac{1}{2}+\frac{z}{4}+\frac{\bar z}{4}) = \cos^2\left(\frac{\theta}{2}\right)$ for the Jacobi polynomials $\tilde{\mathbf{P}}^{(1,0)}(x)$, which is thus also the symbol of the Toeplitz part of $ \mathrm{D}^{-1} \mathrm{J} \mathrm{D} $. Note at this point that $$ \left(\mathrm{D}^{-1} \mathrm{J} \mathrm{D} \right)^{n-j} =  \mathrm{D}^{-1} \mathrm{J}^{n-j} \mathrm{D}$$ due to the outer operators cancelling. Given these tools as well as the linearity of the Fourier series it follows that the symbol of the Toeplitz part of the Volterra operator $\tilde{\mathrm{V}}_K$ is the linear combination
$$ f(z) = \sum_{n=0}^M \sum_{j=0}^n k_{nj} \cos^{2n}\left(\frac{\theta}{2}\right).$$\end{proof}
\begin{theorem}[] The method described in Section \ref{sec:methodvolterraIE} converges for well-posed Volterra integral equations of the first kind with limits of integration $0$ to $1-x$
$$ \mathrm{V}_K \mathbf{u} = \mathbf{g}, $$
rewritten as
$$ \tilde{\mathrm{V}}_K \mathbf{u} = \mathbf{q}, $$
with $q(x) = \frac{g(x)}{1-x}$ for a polynomial kernel $K(x,y) \in L^2[T^2]$ and with $\mathbf{q} \in \ell^2_1$, subject to the symbol of the Toeplitz part of $\tilde{\mathrm{V}}_K$ not vanishing on the complex unit circle. This condition is fulfilled if and only if $\forall x \in [0,1]: K(x,x)\neq 0$. \end{theorem}

\begin{proof} The requirement $\mathbf{q} \in \ell^2_1$ arises formally due to the need to first invert $\mathrm{D}$ and can be understood as stemming from the inverse integration being a differentiation. The invertibility conditions of asymptotically Toeplitz operators of the form $(\mathrm{T}+\mathcal{K})$ are known in the literature (see e.g. \cite{hagen_c*-algebras_2001,bottcher_analysis_2006} and the references therein): A compactly perturbed Toeplitz operator on $\ell^2$ is invertible if it is a Fredholm operator, its index is $0$ and it has a trivial kernel \cite{grobler_operator_2010,bottcher_analysis_2006,hagen_c*-algebras_2001}. Furthermore, a compactly perturbed Toeplitz operator is Fredholm if its symbol (which is just the symbol of the Toeplitz part) does not vanish anywhere on the complex unit circle. 

In general, it holds that the index of a Toeplitz operator which is Fredholm is the sign-flipped winding number of its symbol on the complex unit disk \cite{bottcher_analysis_2006}. Since the symbol of the Toeplitz part of the unweighted Volterra operator is real-valued and continuous its index is thus 0 if and only if it does not vanish anywhere on the complex unit circle, which is a necessary condition for it to be Fredholm in the first place. Since $\cos^{2}\left(\frac{\theta}{2}\right) \in [0,1]$, the symbol vanishes at some point $\theta \in [0,2 \pi]$, i.e. 
$$\sum_{n=0}^M \sum_{j=0}^n k_{nj} \cos^{2n}\left(\frac{\theta}{2}\right) = 0,$$
if and only if for some $x \in [0,1]$ we have
$$\sum_{n=0}^M \sum_{j=0}^n k_{nj} x^n = 0.$$
This in turn is precisely the condition that $K(x,x) = 0$, since 
$$
K(x,y)=\sum_{n=0}^M \sum_{j=0}^n k_{nj} x^{n-j} y^{j}.
$$ Conversely, if $\forall x \in [0,1]: K(x,x)\neq 0$ then the Volterra operator is Fredholm because the symbol of its Toeplitz part has no roots on the unit circle and as this symbol is real valued its winding number and thus index is 0. This necessary condition for invertibility of the operator becomes a sufficient condition if in addition to this we have $\text{ker}(\mathrm{T}+\mathcal{K}) = \{ 0\}$, as this yields injectivity and via the index formula \cite{bottcher_analysis_2006}: $$\text{ind(T)}=\text{ind}(\mathrm{T}+\mathcal{K}) := \text{dim}(\text{ker}(\mathrm{T}+\mathcal{K})) - \text{dim}(\text{coker}(\mathrm{T}+\mathcal{K})),$$ with $\text{ind}(\mathrm{T}+\mathcal{K})=0$ also implies surjectivity. $\text{ker}(\mathrm{T}+\mathcal{K}) = \{ 0\}$ is a consequence of the classical result that the Volterra integral operator has no non-zero eigenvalues. The convergence of the method is then a consequence of known results in the theory of finite section methods, see e.g. \cite{hagen_c*-algebras_2001}.  \end{proof}

\medskip
\noindent{\bf Remark}: 
The motivation for solving $ \tilde{\mathrm{V}}_K \mathbf{u}=\mathbf{q}$ with $q(x) = \frac{g(x)}{1-x}$ instead of $ \mathrm{V}_K \mathbf{u}=\mathbf{g}$ directly can be understood at this point, since for $\mathrm{V}_K$ the symbol of the Toeplitz part is instead found to be $$\sum_{n=0}^M \sum_{j=0}^n k_{nj} \sin\left(\frac{\theta}{2}\right) \cos^{2n}\left(\frac{\theta}{2}\right),$$
which always has a root on the complex unit circle at $\theta = 0$ and thus its induced Toeplitz operator is not Fredholm and not invertible. Therefore the presented proof strategy only succeeds if $q(x) = \frac{g(x)}{1-x}$ may be used instead to get rid of the additional sine terms. The symbol of the Toeplitz part of $\tilde{\mathrm{V}}_K$ is comparably very well-behaved for a variety of kernels. 
\medskip

So far we have only been working with polynomial kernels of order $M$, henceforth denoted $K_M$, when it comes to Volterra equations of the first kind. We will need the following theorem (see \cite{atkinson_theoretical_2009,trogdon_riemann-hilbert_2016})  which we restate without proof for the extension of the above arguments to a non-polynomial kernel:

\begin{theorem}\label{theorem:HPbook}
Let $X$ and $Y$ be normed linear spaces with one or both being Banach spaces and let $\mathcal{T} : X \rightarrow Y$ be a bounded and invertible operator with $\mathcal{T}^{-1} : Y \rightarrow X$. Then if the bounded operator $\mathcal{M} : X \rightarrow Y$ satisfies
$$\| \mathcal{M} - \mathcal{T} \| < \frac{1}{ \| \mathcal{T}^{-1} \|},$$
it follows that $\mathcal{M}$ is also invertible with bounded inverse operator $\mathcal{M}^{-1}: Y \rightarrow X$ and
$$ \|\mathcal{M}^{-1} \| \leq \frac{\| \mathcal{T}^{-1} \|}{1- \| \mathcal{T}^{-1} \|\| \mathcal{T} - \mathcal{M} \|},  $$
$$ \|\mathcal{M}^{-1} - \mathcal{T}^{-1}\| \leq \frac{\| \mathcal{T}^{-1} \|^2 \| \mathcal{T} - \mathcal{M} \|}{1- \| \mathcal{T}^{-1} \|\| \mathcal{T} - \mathcal{M} \|}. $$
\end{theorem}

\begin{lemma}
Given that $$ \| \tilde{\mathrm{V}}_{K_M} - \tilde{\mathrm{V}}_{K}\| \xrightarrow[M\rightarrow\infty]{}0$$ for a sequence of Volterra operators induced by polynomial kernels $K_M (x,y)$ and a not necessarily polynomial kernel $K(x,y)$, we have $$ \| \mathbf{u}_M - \mathbf{u} \| \xrightarrow[M\rightarrow\infty]{}0,$$ where $\mathbf{u}_M$ is the solution to the approximated problem
$$ \tilde{\mathrm{V}}_{K_M}\mathbf{u}_M = \mathbf{q}.$$
\end{lemma}
\begin{proof}
The method can be extended to more general $K=K(x,y)$ if $K_M$ is interpreted as the polynomial approximation of order $M$ of the full kernel $K$. To show that the method can be extended sensibly to non-polynomial kernels what remains to be shown is that $ \| \mathbf{u}_M - \mathbf{u} \| \xrightarrow[M\rightarrow\infty]{}0.$ This can be achieved by use of Theorem \ref{theorem:HPbook}: The assumptions of the theorem are satisfied when setting $\mathcal{T} =  \tilde{\mathrm{V}}_{K}$ and $\mathcal{M} =  \tilde{\mathrm{V}}_{K_M}$ since if $ \| \tilde{\mathrm{V}}_{K_M} - \tilde{\mathrm{V}}_{K}\| \xrightarrow[M\rightarrow\infty]{}0$ then for some $M$ all subsequent $\tilde{\mathrm{V}}_{K_M}$ satisfy $$ \| \tilde{\mathrm{V}}_{K_M} - \tilde{\mathrm{V}}_{K}\| < \frac{1}{\| \tilde{\mathrm{V}}_{K}^{-1} \|}.$$
This immediately yields invertibility of $\tilde{\mathrm{V}}_{K_M}$ and more importantly the desired result that
$$ \| \tilde{\mathrm{V}}_{K_M}^{-1} - \tilde{\mathrm{V}}_K^{-1} \| < \frac{\| \tilde{\mathrm{V}}^{-1} \|^2 \| \tilde{\mathrm{V}}_{K_M} - \tilde{\mathrm{V}}_{K}\| }{1-\| \tilde{\mathrm{V}}^{-1} \|\| \tilde{\mathrm{V}}_{K_M} - \tilde{\mathrm{V}}_{K}\|}\xrightarrow[M\rightarrow\infty]{}0 $$
which justifies calling the solution $\mathbf{u}_M= \tilde{\mathrm{V}}_{K_M}^{-1} \mathbf{q}$ an approximation to $\mathbf{u}= \tilde{\mathrm{V}}_{K}^{-1} \mathbf{q}$.
\end{proof}

\section{Discussion} \label{sec:discussion}
The method proposed in this paper can efficiently compute Volterra integrals as well as solve Volterra integral equations of the first and second kind with high accuracy using bivariate orthogonal polynomials to resolve the kernel along with an operator valued Clenshaw algorithm and is not restricted to convolution kernels. Numerical experiments suggest it can even be applicable to certain singular equations. Our approach takes advantage of the sparsity of the required integration and extension operators which are due to the symmetries of the Jacobi polynomial basis on the triangle domain. The method was shown to converge for well-posed Volterra integral equations of the first and second kind, using a link to compact perturbations of Toeplitz operators. 

Extensions of this approach to various so-called integro-differential equations of Volterra type, where both differentiation and Volterra operators act on the unknown function, as well as extensions to non-linear Volterra equations, where the unknown function can appear in non-linear fashion in the Volterra integral, while non-trivial are conceivable and will be addressed in future works.

\section*{Acknowledgments}

We thank Nick Hale and Kuan Xu for crucial help on Volterra integral equations at the initial stages of this project.  We thank Mikael Slevinsky for thoroughly reading a draft and providing detailed comments.

\bibliographystyle{plain}
\bibliography{citations}

\end{document}